\documentclass[11pt]{amsart}
\usepackage{amsmath,a4wide,scalerel}
\usepackage{mathabx}
\usepackage{xcolor,graphicx}
\usepackage{mathrsfs}
\usepackage{pgfplots}
\usepackage{subcaption}
\usepackage[normalem]{ulem}
\usepackage{float}
\usepackage{accents}
\usepackage{bbold}
\usepackage{enumitem}

\usepackage{hyperref}
\newtheorem{theorem}{Theorem}
\newtheorem{lemma}[theorem]{Lemma}
\newtheorem{proposition}[theorem]{Proposition}

\newtheorem{corollary}[theorem]{Corollary}

\newtheorem{assumption}{Assumption}

\newcommand{\N}{{\mathbb N}}
\newcommand{\E}{{\mathbb E}}
\newcommand{\PP}{{\mathbb P}}

\newcommand*\diff{\mathop{}\!\mathrm{d}}
\newcommand{\D}{\mathcal{D}}

\newcommand{\CD}{\operatorname{cl} \D}
\newcommand{\app}{\operatorname{approx}}
\newcommand{\refe}{\operatorname{ref}}
\newcommand{\aux}{\operatorname{aux}}

\newcommand{\C}{\mathcal{C}}
\newcommand{\EM}{\textrm{EM}}
\newcommand{\MIL}{\textrm{MIL}}
\newcommand{\IT}{\textrm{IT}}
\newcommand{\LIT}{\textrm{LIT}}
\newcommand{\SEM}{\textrm{SEM}}
\newcommand{\TE}{\textrm{TE}}
\newcommand{\LEM}{\textrm{LEM}}
\newcommand{\Lo}{\mathcal{L}^{0}}
\newcommand{\Li}{\mathcal{L}^{1}}
\newcommand{\LL}{\mathcal{L}}
\newcommand{\id}{\operatorname{id}}
\newcommand{\Lip}{\operatorname{Lip}}

\usepackage{color}
\usepackage{dsfont}

\title[]{Boundary-preserving Lamperti--Itô--Taylor approximations for some stochastic differential equations.}

\date{\today}

\author{Johan Ulander}
              \address{Swedish Defence Research Agency (FOI), 58330 Linköping, Sweden}
              \email{\tt johan.ulander@foi.se}
              
\begin{document}

\begin{abstract}
In this work, we propose high-order boundary-preserving numerical schemes for the strong approximation for some scalar stochastic differential equations with invariant domains being open and bounded intervals. The proposed methods involve using the Lamperti transform to map the SDE to another SDE with additive noise with a trivial invariant domain. Then, by imposing regularity assumptions on the original coefficient functions, we can guarantee that the drift coefficient function of the transformed SDE is regular, and known high-order schemes can be used to achieve the desired convergence order. We confirm the theoretical results with numerical experiments.
\end{abstract}

\maketitle

{\small\noindent


\vspace{1cm}

\section{Introduction}\label{sec:intro}
Stochastic differential equations (SDEs) are used to model a wide range of physical and non-physical phenomena. This includes various phenomena in finance, biology, and physics, to name but a few \cite{MR3363443,GraySIS, Karlin1981ASC,KERMACK199133, introTostocCalc, MR1214374, MR496534,MR2001996}. These models typically do not have closed-form solutions, which means that numerical approximations are crucial to their use. The literature on general numerical schemes to approximate the solutions of SDEs is by now well-established. In recent years, interest in numerical schemes that, in addition to converging, preserve certain behaviour or properties of the underlying SDE has risen. These types of numerical schemes belong to the wider class of structure-preserving methods \cite{MR2221614}. In this work, we consider scalar stochastic differential equations (SDEs) whose solutions remain within a bounded domain for all time. Such a domain is called an invariant domain: whenever the initial condition belongs to the domain, the corresponding solution remains in the domain for all later times. We propose numerical schemes that preserve this invariant domain and achieve arbitrarily high strong convergence order, provided that the coefficients of the SDE are sufficiently regular. Examples of such SDEs include, but are not limited to, Allen--Cahn type SDEs, SIS-type SDEs, and Nagumo-type SDEs, all of which we provide numerical experiments for in Section~\ref{sec:numExp}. Many works have shown that general-purpose and classical numerical schemes to approximate the solutions of SDEs do not preserve invariant domains \cite{Ulander2024AB,MR4737060}. Typically, failure to preserve the invariant domain arises from the discretisation of the noise, which introduces random variables with unbounded support on the real line. 

Several approaches have been developed for constructing numerical schemes that preserve the invariant domains of SDEs. These include implicit \cite{MR2931351,MR4850627,MR1410392}, transformation-based \cite{MR3006996,MR4220738,MR2898556,MR3248050,MR4274899}, truncation-based \cite{MR2367990,MR4242953}, geometric Brownian motion (gBM)-based \cite{MR4177372,MR4780408,MR4729657,DomPres}, and time-splitting \cite{MR4780408,MR4544037,MR2341800} methods. We mention the PhD thesis of the author \cite{MR5051815} that discusses each of these methods in detail. The class of methods most relevant to the present work is that of Lamperti-based schemes, a subclass of transformation-based methods in which the Lamperti transform is applied to the original SDE. This transformation converts the SDE into one with additive noise, typically at the expense of a more complicated drift coefficient. It is particularly well suited to the construction of boundary-preserving numerical schemes, since it usually defines a bijection between the invariant domain of the original SDE and a transformed domain on which the resulting SDE is discretised. Consequently, any numerical approximation that remains within the transformed domain is mapped by the inverse Lamperti transform to an approximation that remains within the invariant domain of the original SDE. Recent research has therefore focused on developing Lamperti-based schemes that achieve high-order strong convergence, are applicable to broad classes of SDEs, and are straightforward to implement. We mention recent works \cite{MR4888024,MR4597411} that develop Lamperti-based schemes with convergence orders exceeding the strong order $1$ typically achieved by such methods.

In numerical stochastic analysis, the Euler--Maruyama (EM) scheme is the natural analogue of the forward Euler method for ordinary differential equations. Under standard global Lipschitz and regularity assumptions, the EM scheme converges strongly with order $1/2$. In the case of additive noise, this order can be improved to $1$, provided that the drift coefficient satisfies suitable regularity assumptions. For strong convergence of high-orders, Itô--Taylor schemes provide a natural class of methods. These schemes are derived by truncating the Itô--Taylor expansion of the solution, and their convergence order is determined by the level at which this expansion is truncated. Consequently, if the coefficient functions are sufficiently smooth and possess bounded derivatives of the required orders, numerical schemes of arbitrarily high strong order can be constructed in this way. However, as is the case for many classical discretisation methods, standard Itô--Taylor schemes do not, in general, preserve invariant domains. Modified Itô--Taylor schemes that achieve arbitrarily high strong convergence orders while approximately preserving the invariant domain were introduced in \cite{MR2481529}. The present work addresses the issue of exact domain preservation by constructing Lamperti-based numerical schemes which, under sufficient regularity assumptions on the coefficient functions, attain arbitrarily high strong order while preserving the invariant domain exactly for a class of SDEs. This class is closely related to that studied in \cite{MR4737060}, but we impose stronger regularity assumptions in order to obtain strong convergence of high-order.

The main contributions of this paper are the following:
\begin{itemize}
[
    topsep=3pt,
    itemsep=0pt,
    parsep=0pt,
    partopsep=0pt
]
\item We propose a family of boundary-preserving numerical schemes for the strong approximation for a family of SDEs with a bounded invariant domain.
\item We prove the boundary-preserving property and strong convergence order that depends on the regularity of the coefficient functions, see Corollary~\ref{cor:LIT_conv}.
\item We numerically verify the boundary-preserving property and the strong convergence order obtained theoretically, see Section~\ref{sec:numExp}.
\end{itemize}
To the best of our knowledge, these are the first boundary-preserving numerical schemes for SDEs with solutions in a bounded domain that achieve strong convergence orders higher than $1.5$. 

The paper is structured as follows. First, we introduce the setting for the work. Next, in the main section of this work, we define the Lamperti--Itô--Taylor schemes of arbitrarily high order (of strong convergence) and prove the main strong convergence theorem. This section also includes an overview of Itô--Taylor expansions and schemes, with a particular focus on SDEs with additive noise, as this is the relevant case for the Lamperti--Itô--Taylor schemes. Here we also discuss the computation and approximation of the iterated integrals showing up in the Itô--Taylor schemes needed for the numerical computations. Finally, we provide numerical experiments in Section~\ref{sec:numExp} to numerically verify the theoretical results in Section~\ref{sec:num_schemes}.

\section{Setting}\label{sec:setting}
This section introduces the needed notions and notation for the work. Throughout this work, $\mathbb{R}$ denotes the real line and $\mathbb{N} = \{ 1,2,3,\ldots \}$ denotes the natural numbers. We let $(\Omega, \mathcal{F}, \mathbb{P})$ be a fixed probability space equipped with a complete and right-continuous filtration $\left( \mathcal{F}_{t} \right)_{t \in [0, T]}$, and $\E$ denotes the expectation operator. We let $\C^{k}(A)$ denote the space of $k$ times continuously differentiable functions from $A$ to $\mathbb{R}$, and we let $\C_{b}^{k}(A)$ denote the space of $k$ times continuously differentiable functions from $A$ to $\mathbb{R}$ whose derivatives are bounded from order $0$ up to order $k$. A bounded $0$th derivative means that the function is bounded. We use $\Lip_{F}$ and $L_{F}$ to denote the Lipschitz constant and the linear growth constant, respectively, of the function $F$. We let $C$ denote a generic constant that may change from line to line. Most equalities and inequalities are to be understood in the almost sure sense.

We consider scalar Itô stochastic differential equations of the form
\begin{equation}\label{eq:SDE}
\left\lbrace
\begin{aligned}
& \diff X(t) = f(X(t)) \diff t + g(X(t)) \diff B(t),\ t \in (0,T], \\ 
& X(0) = x_{0} \in \D,
\end{aligned}
\right.
\end{equation}
where $T \in (0,\infty)$, $f: \mathbb{R} \to \mathbb{R}$ and $g: \mathbb{R} \to \mathbb{R}$ are functions satisfying some regularity assumptions to be specified in Section~\ref{sec:LampTrans}, $B$ is a standard Brownian motion, and $\D = (a,b) \subset \mathbb{R}$ is an open and bounded set such that
\begin{equation}\label{eq:invDom}
\PP (X(t) \in \D,\ \forall t \in [0,T]) = 1.
\end{equation}
We emphasise that the assumptions on $f$ and $g$ specified in Section~\ref{sec:LampTrans} imply that~\eqref{eq:invDom} is satisfied.

A (strong) solution $X$ to~\eqref{eq:SDE} is to be understood as a stochastic process satisfying the following integral equation
\begin{equation*}
X(t) = x_{0} + \int_{0}^{t} f(X(s)) \diff s + \int_{0}^{t} g(X(s)) \diff B(s),\ t \in [0,T],
\end{equation*}
where $\left( B(t) \right)_{t \in [0,T]}$ is a given Brownian motion.

\section{The Lamperti transform}\label{sec:LampTrans}
We impose the following assumptions on $f$ and $g$, which generalise the assumptions from \cite{MR4737060}. We fix a regularity parameter $k \in \{2,3,4,\ldots\}$, which determines the differentiability requirements imposed on the coefficient functions $f$ and $g$ in the assumptions below.
\begin{assumption}\label{ass:f}
\emph{The drift coefficient} $f \in \C^{k} \left( \CD \right)$.
\end{assumption}
\begin{assumption}\label{ass:g}
\emph{The diffusion coefficient $g \in \C^{k+1} \left(\CD \right) $ and is strictly positive on $\D = (a,b)$,  and the following non-integrability conditions are satisfied}
\begin{equation}\label{eq:intCond}
\int_{w_{0}}^{a} \frac{1}{g(w)} \diff w = - \infty,\ \int_{w_{0}}^{b} \frac{1}{g(w)} \diff w = \infty,
\end{equation}
\emph{for any $w_{0} \in \D$.}
\end{assumption}
\begin{assumption}\label{ass:fg}
\emph{The drift coefficient $f$ decays at least as fast as the diffusion coefficient $g$ near the boundary points $\partial \D$; that is, the following limits exist and are finite}
\begin{equation}\label{eq:limCond}
\left| \lim_{r \searrow a} \frac{f(r)}{g(r)} \right| + \left| \lim_{x \nearrow b} \frac{f(r)}{g(r)} \right| < \infty.
\end{equation}
\end{assumption}
We next provide the prototypical example of coefficient functions $f$ and $g$ that satisfy Assumptions~\ref{ass:f},~\ref{ass:g}, and~\ref{ass:fg}, and it is taken from \cite{MR4737060}. Let $g$ be of the form
\begin{equation*}
g(r) = (r - a)^{\beta_{a}} (r - b)^{\beta_{b}} \tilde{g}(r),\ r \in [a,b],
\end{equation*}
where $\beta_{a}, \beta_{b} \in \{ 1,2,3,\ldots \}$ are the multiplicities of the roots $r=a$ and $r=b$, respectively, and where $\tilde{g}$ is some polynomial with no roots in $[a,b]$. Next, let $f$ be of the form
\begin{equation*}
f(r) = (r-a)^{\delta_{a}} (r-b)^{\delta_{b}} \tilde{f}(r),\ r \in [a,b],
\end{equation*}
where $\delta_{a}, \delta_{b} \in \{ 1,2,3,\ldots \}$ are the multiplicities of the roots $r=a$ and $r=b$, respectively, and where $\tilde{f}$ is some polynomial with no roots in $\CD = [a,b]$. If $1 \leq \beta_{a} \leq \delta_{a}$ and if $1 \leq \beta_{b} \leq \delta_{b}$, then $f$ and $g$ satisfy Assumptions~\ref{ass:f},~\ref{ass:g}, and~\ref{ass:fg} with $\D = (a,b)$.

Next, we will motivate the above assumptions. The Lamperti transform~\cite{MR176536} is defined by
\begin{equation}\label{eq:LampTrans}
\Phi(r) = \int_{w_{0}}^{r} \frac{1}{g(w)} \diff w,\ r \in \D,
\end{equation}
for some $w_{0} \in (a,b)$, and is constructed in such a way that the process $Y(t) = \Phi(X(t))$, where $X(t)$ is the solution of~\eqref{eq:SDE}, satisfies
\begin{equation}\label{eq:SDE-Lamp}
\left\lbrace
\begin{aligned}
& \diff Y(t) = H(Y(t)) \diff t + \diff B(t),\ t \in [0,T], \\
& Y(0) = \Phi(x_{0}),
\end{aligned}
\right.
\end{equation}
where $H$ is given by
\begin{equation}\label{eq:Lamp_drift}
H(r) = \frac{f(\Phi^{-1}(r))}{g(\Phi^{-1}(r))} - \frac{1}{2} g'(\Phi^{-1}(r)),\ r \in \mathbb{R}.
\end{equation}
The following proposition, Proposition~\ref{prop:phiinv}, is used to transfer error estimates derived for approximations of $Y$ in~\eqref{eq:SDE-Lamp} to error estimates for approximations of the solution $X$ to~\eqref{eq:SDE}. We refer to the proof of Proposition~$1$ in~\cite{MR4737060} for a proof of Proposition~\ref{prop:phiinv}.
\begin{proposition}\label{prop:phiinv}
Suppose Assumption~\ref{ass:g} is satisfied. Then $\Phi^{-1}: \mathbb{R} \to \D$ is bounded, bijective, continuously differentiable, and has a bounded derivative. In particular, $\Phi^{-1}: \mathbb{R} \to \D$ is globally Lipschitz continuous, and we denote the Lipschitz constant of $\Phi^{-1}$ by $\Lip_{\Phi^{-1}}$.
\end{proposition}
We remark that $\Phi^{-1}$ actually has high-order derivatives, but Proposition~\ref{prop:phiinv} is sufficient for our purposes.

Next, the regularity assumptions imposed on $f$ and $g$ imply that the drift coefficient $H$ in~\eqref{eq:SDE-Lamp} has bounded derivatives up to order $k$. This result, stated in Proposition~\ref{prop:H}, is a key ingredient in the construction of the high-order schemes developed in this work.
\begin{proposition}\label{prop:H}
Suppose Assumptions~\ref{ass:f},~\ref{ass:g}, and~\ref{ass:fg} are satisfied. Then $H \in \C^{k}_{b}(\mathbb{R})$.
\end{proposition}
The proof of Proposition~\ref{prop:H} is the high-order extension of Proposition~$2$ in~\cite{MR4737060}, and consists of computing high-order derivatives of $H$. Proposition~\ref{prop:H} implies, in particular, that the solution $Y$ of the SDE in~\eqref{eq:SDE-Lamp} exists, is unique, and satisfies
\begin{equation}\label{eq:Yt_bounded}
\PP \left( Y(t) \in \mathbb{R},\ \forall t \in [0,T] \right) = 1.
\end{equation}
Indeed, by applying the triangle inequality to (the integral version of)~\eqref{eq:SDE-Lamp}, we have the bound
\begin{equation*}
\sup_{t \in [0,T]} |Y(t)| \leq C + \sup_{t \in [0,T]} |B(t)| < \infty,
\end{equation*}
almost surely, since $H \in \C^{k}_{b}(\mathbb{R})$ is bounded by Proposition~\ref{prop:H}. Furthermore, combining this with Proposition~\ref{prop:phiinv} gives us that the solution $X = \Phi^{-1}(Y)$ of the SDE in~\eqref{eq:SDE} exists, is unique, and satisfies
\begin{equation*}
\PP \left( X(t) \in \D,\ \forall t \in [0,T] \right) = 1.
\end{equation*}

The idea of the proposed schemes is to approximate the solution of~\eqref{eq:SDE-Lamp}, and then apply the inverse of the Lamperti transform to obtain an approximation of the solution $X$ of~\eqref{eq:SDE}. To this end, we introduce a discretisation parameter $M \in \mathbb{N}$, and partition the time interval into $[t_{m},t_{m+1}]$, for $m=0,\ldots,M$, each of size $\Delta t = T/M$. We impose the following assumptions in order to state the strong convergence result of this section.
\begin{assumption}\label{ass:globalStrongInt}
There exists an approximating sequence $Y_{0},\ldots,Y_{M}$ of the solution $Y$ of~\eqref{eq:SDE-Lamp} on the time grid $t_{0},\ldots,t_{M}$ such that
\begin{equation*}
\PP (Y_{m} \in \mathbb{R},\ \forall m = 0,\ldots,M) = 1,
\end{equation*}
and, for every $p \in (0,\infty)$, is $p$-strongly convergent with order $\nu$:
\begin{equation*}
\left( \E \left[ \sup_{m=0, \ldots, M} | Y_{m} - Y(t_{m}) |^{p} \right] \right)^{\frac{1}{p}} \leq C \Delta t^{\nu},
\end{equation*}
for some $\nu>0$ and some constant $C>0$ independent of $\Delta t$.
\end{assumption}
Of particular interest to us is that some Itô--Taylor-$\gamma$ schemes (see Section~\ref{sec:IT_general}) satisfy Assumption~\ref{ass:globalStrongInt} (see Proposition~\ref{prop:IT_conv_add}). 

The following is our main general theorem.
\begin{theorem}\label{theo:mainStrong}
Let $M\in\N$, $T > 0$, $\Delta t = T/M$ and let $x_{0} \in \D$.
Suppose Assumptions~\ref{ass:f},~\ref{ass:g},~\ref{ass:fg}, and ~\ref{ass:globalStrongInt} are satisfied. Let $X_{m} = \Phi^{-1}(Y_{m})$, for $m=0,\ldots,M$, where $Y_{m}$ is defined by Assumption~\ref{ass:globalStrongInt}, and let $X$ be the exact solution of the considered SDE in equation~\eqref{eq:SDE}. Then
\begin{equation*}
\PP (X_{m} \in \D,\ \forall m = 0,\ldots,M) = 1,
\end{equation*}
and, for every $p \in (0,\infty)$, it holds
\begin{equation*}
\left( \E \left[ \sup_{m=0, \ldots, M} | X_{m} - X(t_{m}) |^{p} \right] \right)^{\frac{1}{p}} \leq C \Delta t^{\nu},
\end{equation*}
where the constant $C>0$ does not depend on $\Delta t$.
\end{theorem}
\begin{proof}
The first statement follows from $X_{m} = \Phi^{-1}(Y_{m})$, for $m=0,\ldots,M$, that $\Phi$ is a bijection, and the first property in Assumption~\ref{ass:globalStrongInt}. For the error estimate, let us first assume that $p \in [2,\infty)$. We reduce the statement to the corresponding one for $Y$
\begin{equation*}
| X_{m} - X(t_{m}) | = | \Phi^{-1}(Y_{m}) - \Phi^{-1}(Y(t_{m})) | \leq \Lip_{\Phi^{-1}} |Y_{m} - Y(t_{m}) |,
\end{equation*}
and using Assumption~\ref{ass:globalStrongInt} to obtain the desired estimate
\begin{align*}
\E \left[ \sup_{m=0, \ldots, M} | X_{m} - X(t_{m}) |^{p} \right] &\leq \Lip_{\Phi^{-1}}^{p} \E \left[ \sup_{m=0, \ldots, M} | Y_{m} - Y(t_{m}) |^{p} \right] \\ &\leq C \Delta t^{p \nu}.
\end{align*}
The case $p \in (0,2)$ follows from the $p=2$ case combined with Hölder's inequality.
\end{proof}

\section{Lamperti--Itô--Taylor schemes}\label{sec:num_schemes}
In this section, we construct boundary-preserving numerical schemes, that we refer to as Lamperti--Ito--Taylor schemes, to approximate the solution $X$ of~\eqref{eq:SDE} based on the following:
\begin{enumerate}
\item \textbf{Use the Lamperti transform $\Phi$:} The transformed process $Y(t) = \Phi(X(t))$ satisfies 
\begin{equation}\label{eq:LITsdeAdd}
\left\lbrace
\begin{aligned}
& \diff Y(t) = H(Y(t)) \diff t + \diff B(t),\ t \in [0,T],\\
& Y(0) = \Phi(x_{0}).
\end{aligned}
\right.
\end{equation}
\item \textbf{Apply the Itô--Taylor scheme of strong order $\gamma$:} Approximate the solution $Y$ of~\eqref{eq:LITsdeAdd} on the time grid $0=t_{0},\ldots,t_{M}=T$ using the Itô--Taylor scheme $Y_{0},\ldots,Y_{M}$ of strong convergence order $\gamma$.
\item \textbf{Use the inverse Lamperti transform $\Phi^{-1}$:} $\Phi^{-1}(Y_{0}),\ldots,\Phi^{-1}(Y_{M})$ is an approximation $X$ on the time grid $0=t_{0},\ldots,t_{M}=T$ that is boundary-preserving and $p$-strongly convergent of order $\gamma$.
\end{enumerate}
We next discuss $(2)$, since $(1)$ and $(3)$ are self-explanatory. In the following, we give an overview of how Itô--Taylor schemes are constructed, and how they simplify for the SDE with additive noise in~\eqref{eq:LITsdeAdd}.

\subsection{General Itô--Taylor schemes}\label{sec:IT_general}
The Itô--Taylor schemes are the generalisations of Taylor schemes for ODEs to SDEs, and they differ because Taylor expansions in differential calculus are replaced with Itô's lemma in stochastic calculus. Similarly to Taylor expansions, we will use integration variables $s, s_{1},s_{2},\ldots$ to not confuse them with the time grid $t_{0},\ldots,t_{M}$. The content presented in this section can be found in classical books on numerical methods for SDEs, for example \cite{MR1214374}.

Itô--Taylor schemes are obtained by truncating Itô--Taylor expansions at certain levels. We start by describing Itô--Taylor expansions, and then we show how the truncation is done. Let us consider an SDE of the form
\begin{equation}\label{eq:SDE_IT_gen}
\left\lbrace
\begin{aligned}
& \diff Z(s) = \mu(Z(s)) \diff s + \sigma(Z(s)) \diff B(s),\ s \in [0,T],\\
& Z(0) = z_{0} \in \mathbb{R},
\end{aligned}
\right.
\end{equation}
with $\mu,\sigma: \mathbb{R} \to \mathbb{R}$ sufficiently regular, or equivalently in integrated form
\begin{equation*}
Z(s) = z_{0} + \int_{0}^{s} \mu(Z(s_{1}))) \diff s_{1} + \int_{0}^{s} \sigma(Z(s_{1})) \diff B(s_{1}),\ s \in [0,T].
\end{equation*}
We use $\mu$ and $\sigma$ to denote the drift and diffusion coefficients and $Z(s)$ to denote the solution process in~\eqref{eq:SDE_IT_gen} to not confuse it with~\eqref{eq:SDE}. If not otherwise stated, we assume that $s \in [0,T]$ and that $r \in \mathbb{R}$ in this section.

The key ingredient in Itô--Taylor expansions is Itô's lemma that states that the process $F(Z(s))$, for $F : \mathbb{R} \to \mathbb{R}$ being sufficiently smooth, satisfies the following SDE
\begin{align*}
\diff F(Z(s)) &= \left( \mu(Z(s)) F'(Z(s)) + \frac{1}{2} \sigma^{2}(Z(s)) F''(Z(s)) \right) \diff s + \sigma(Z(s)) F'(X(s)) \diff B(s) \\ &= \mathcal{L}^{0}[F](Z(s)) \diff s + \mathcal{L}^{1}[F](Z(s)) \diff B(s),
\end{align*}
where we introduced
\begin{equation}\label{eq:LoDef}
\mathcal{L}^{0}[F](r) = \mu(r) F'(r) + \frac{1}{2} g^{2}(r) F''(r)
\end{equation}
and
\begin{equation}\label{eq:LiDef}
\mathcal{L}^{1}[F](r) = \sigma(r) F'(r).
\end{equation}
We integrate the above over $[s_{0},s] \subset [0,T]$ to obtain
\begin{equation}\label{eq:ItoTaylorF}
F(Z(s)) = F(Z(s_{0})) + \int_{s_{0}}^{s} \mathcal{L}^{0}[F](Z(s_{1})) \diff s_{1} + \int_{s_{0}}^{s} \mathcal{L}^{1}[F](Z(s_{1})) \diff B(s_{1}),
\end{equation}
which forms the basis of the Itô--Taylor expansion. The time point $s_{0}$ is the reference point around which we expand the solution $Z(s)$, for $s \geq s_{0}$.

Next, we apply~\eqref{eq:ItoTaylorF} to $\Lo[F](Z(s_{1}))$  to obtain
\begin{equation*}
\Lo[F](Z(s_{1})) = \Lo[F](Z(s_{0})) + \int_{s_{0}}^{s_{1}} \Lo \Lo[F](Z(s_{2})) \diff s_{2} + \int_{s_{0}}^{s_{1}} \Li \Lo[F](Z(s_{2})) \diff B(s_{2})
\end{equation*}
and to $\Li[F](Z(s_{1}))$ to obtain
\begin{equation*}
\Li[F](Z(s_{1})) = \Li[F](Z(s_{0})) + \int_{s_{0}}^{s_{1}} \Lo \Li[F](Z(s_{2})) \diff s_{2} + \int_{t_{0}}^{s_{1}} \Li \Li[F](Z(s_{2})) \diff B(s_{2}).
\end{equation*}
Thus, inserting these formulas back into~\eqref{eq:ItoTaylorF} gives
\begin{equation}\label{eq:EM_scheme_with_RM}
  \begin{split}
    F(Z(s)) &= F(Z(s_{0})) + \Lo[F](Z(s_{0})) (s - s_{0}) + \Li[F](Z(s_{0})) (B(s)-B(s_{0})) \\ &+ \int_{s_{0}}^{s} \int_{s_{0}}^{s_{1}} \Lo \Lo[F](Z(s_{2})) \diff s_{2} \diff s_{1} \\ & + \int_{s_{0}}^{s} \int_{s_{0}}^{s_{1}} \Li \Lo[F](Z(s_{2})) \diff B(s_{2}) \diff s_{1} \\ &+ \int_{s_{0}}^{s} \int_{s_{0}}^{s_{1}} \Lo \Li[F](Z(s_{2})) \diff s_{2} \diff B(s_{1}) \\ & + \int_{s_{0}}^{s} \int_{s_{0}}^{s_{1}} \Li \Li[F](Z(s_{2})) \diff B(s_{2}) \diff B(s_{1}).
  \end{split}
\end{equation}
We refer to~\eqref{eq:EM_scheme_with_RM} as the Ito--Taylor expansion of order $1/2$. In principle, we could repeat the above strategy to obtain Itô--Taylor expansions of arbitrarily high order (provided that $\mu$ and $\sigma$ are regular enough). We can, however, write~\eqref{eq:EM_scheme_with_RM} in a more compact way, and that makes the repeated use of Itô's lemma in~\eqref{eq:ItoTaylorF} less cumbersome. We introduce multi-indexed functions using the recursion
\begin{equation}\label{eq:DefMultIndCoefFunc}
F_{\alpha} = \LL^{\alpha_{1}} [F_{-\alpha}],
\end{equation}
and multi-indexed integrals by the following recursion
\begin{equation}\label{eq:DefMultIndInt}
I_{\alpha}[F]_{s,t} = \int_{s}^{t}  I_{\alpha-}[F]_{s,s_{1}} \diff B^{\alpha_{L}}(s_{1}),\ 0 \leq s \leq t \leq T,
\end{equation}
where $\alpha = (\alpha_{1},\ldots,\alpha_{L}) \in \{ 0,1 \}^{L}$, $-\alpha = (\alpha_{2},\ldots,\alpha_{L})$ removes the first index, $\alpha- = (\alpha_{1},\ldots,\alpha_{L-1})$ removes the last index, and where $B^{0}(s_{1}) = s_{1}$ and $B^{1}(s_{1}) = B(s_{1})$. We initialise the above with
\begin{equation*}
F_{\emptyset} = F
\end{equation*}
and
\begin{equation*}
I_{\emptyset} [F]_{s,t} = F(t).
\end{equation*}
For notational convenience, we also introduce the short version
\begin{equation}\label{eq:Ialpha_simple_recur}
I_{\alpha}(s,t) = \int_{s}^{t} I_{\alpha-}(s,s_{1}) \diff B^{\alpha_{L}}(s_{1}),\ 0 \leq s \leq t \leq T,
\end{equation}
initialised with $I_{\emptyset}(s,t) = 1$, for $I_{\alpha}[1]_{s,t}$. The following lemma will be used to show that the $\IT \gamma$ schemes only take values in $\mathbb{R}$.
\begin{lemma}\label{lem:Ialpha_bounded}
For every $0 \leq s \leq t \leq T$ and for every multi-index $\alpha \in \{ 0,1 \}^{L}$,
\begin{equation*}
\PP (I_{\alpha}(s,t) \in \mathbb{R}) = 1.
\end{equation*}
\end{lemma}
\begin{proof}
We prove that
\begin{equation*}
\E \left[ \left| I_{\alpha}(s,t) \right|^{2} \right] < \infty,
\end{equation*}
which implies in particular the statement of the lemma. We prove the statement by induction over the length $\ell(\alpha)$ of the multi-index $\alpha$. The base case is immediate:
\begin{equation*}
I_{\emptyset}(s,t) = 1
\end{equation*}
is bounded. Suppose that
\begin{equation*}
\E \left[ \left| I_{\alpha}(s,t) \right|^{2} \right] < \infty
\end{equation*}
for all multi-indices of length $\ell(\alpha) \leq k$. Let now $\alpha$ be such that $\ell(\alpha) = k+1$. If $\alpha_{k+1} = 0$, then Jensen's inequality for integrals applied to~\eqref{eq:Ialpha_simple_recur} gives us
\begin{equation*}
\E \left[ \left| I_{\alpha}(s,t) \right|^{2} \right] \leq (t-s) \int_{s}^{t} \E \left[ \left| I_{\alpha-}(s,s_{1}) \right|^{2} \right] \diff s_{1} < \infty,
\end{equation*}
where we also used the induction hypothesis on $\alpha-$ (since $\ell(\alpha-) = k$). If $\alpha_{k+1}=1$, then Itô's isometry applied to~\eqref{eq:Ialpha_simple_recur} gives us
\begin{equation*}
\E \left[ \left| I_{\alpha}(s,t) \right|^{2} \right] = \int_{s}^{t} \E \left[ \left| I_{\alpha-}(s,s_{1}) \right|^{2} \right] \diff s_{1} < \infty,
\end{equation*}
where we again used the induction hypothesis for $\alpha-$ (since $\ell(\alpha)=k$). Therefore,
\begin{equation*}
\E \left[ \left| I_{\alpha}(s,t) \right|^{2} \right] < \infty
\end{equation*}
for all multi-indices $\alpha$ of length $\ell(\alpha) \leq k +1$. We conclude that
\begin{equation*}
\E \left[ \left| I_{\alpha}(s,t) \right|^{2} \right] < \infty,
\end{equation*}
and the statement follows.
\end{proof}

Let us unwind the above definition. The terms relevant to~\eqref{eq:EM_scheme_with_RM} are the following
\begin{equation*}
F_{\emptyset} = F,\ F_{(0)} = \Lo[F],\ F_{(1)} = \Li[F],
\end{equation*}
\begin{equation*}
I_{\emptyset}(s,t) = 1,\ I_{(0)}(s,t) = t - s,\ I_{(1)}(s,t) = B(t) - B(s),
\end{equation*}
\begin{equation*}
I_{(0,0)}[F_{(0,0)}(Z)]_{s,t} = \int_{s}^{t} \int_{s}^{t} \Lo \Lo [F](Z(s_{2})) \diff s_{2} \diff s_{1},
\end{equation*}
\begin{equation*}
I_{(1,0)}[F_{(1,0)}(Z)]_{s,t} = \int_{s}^{t} \int_{s}^{t} \Li \Lo [F](Z(s_{0})) \diff B(s_{2}) \diff s_{1},
\end{equation*}
\begin{equation*}
I_{(0,1)}[F_{(0,1)}(Z)]_{s,t} = \int_{s}^{t} \int_{s}^{t} \Lo \Li [F](Z(s_{0})) \diff s_{2} \diff B(t_{1}),
\end{equation*}
and
\begin{equation*}
I_{(1,1)}[F_{(1,1)}(Z)]_{s,t} = \int_{s}^{t} \int_{s}^{t} \Li \Li [F](Z(s_{2})) \diff B(s_{2}) \diff B(s_{1}).
\end{equation*}

Using the above formulas, we may rewrite the expansion in~\eqref{eq:EM_scheme_with_RM} in the more compact form as
\begin{equation}\label{eq:IT_F_EM}
  \begin{split}
    F(Z(s)) &= F_{\emptyset}(Z(s_{0})) I_{\emptyset}(s_{0},s) + F_{(0)}(Z(s_{0})) I_{(0)}(s_{0},s) + F_{(1)}(Z(s_{0})) I_{(1)}(s_{0},s) \\ &+ I_{(0,0)}[F_{(0,0)}(Z)]_{s_{0},s} + I_{(1,0)}[F_{(1,0)}(X)]_{s_{0},s} + I_{(0,1)}[F_{(0,1)}(Z)]_{s_{0},s} \\ &+ I_{(1,1)}[F_{(1,1)}(Z)]_{s_{0},s},\ 0 \leq s_{0} \leq s \leq T.
  \end{split}
\end{equation}
In fact, by introducing the sets
\begin{equation}\label{eq:index_A_EM}
A_{0.5} = \{ \emptyset,\ (0),\ (1) \}
\end{equation}
and
\begin{equation}\label{eq:index_B_EM}
B(A_{0.5}) = \{ \alpha \not\in A_{0.5}:\ -\alpha \in A_{0.5} \} = \{ (0,0),\ (0,1),\ (1,1) \},
\end{equation}
we may write this more concisely as
\begin{equation}\label{eq:IT_EM}
F(Z(s)) = \sum_{\alpha \in A_{0.5}} F_{\alpha}(Z(s_{0})) I_{\alpha}(s_{0},s) + \sum_{\alpha \in B(A_{0.5})} I_{\alpha}[F_{\alpha}(Z)]_{s_{0},s},\ 0 \leq s_{0} \leq s \leq T.
\end{equation}
The first sum in~\eqref{eq:IT_EM} is the main term, or approximating term, and the second sum is the remainder term. $A_{0.5}$ is a so-called hierarchical set and $B(A_{0.5})$ is called the boundary set or remainder set of $A_{0.5}$. A finite set $\Gamma$ is called a hierarchical set if $\emptyset \in \Gamma$ and if $\emptyset \neq \alpha \in \Gamma \Rightarrow -\alpha \in \Gamma$. The set $A_{0.5}$ is the set of multi-indices that correspond to the classical Euler-Maruyama (EM) scheme~\cite{MR71666}. The above sum can seem overly complicated, but it is convenient for Itô-Taylor expansions of high-orders.

We obtained the formula in~\eqref{eq:IT_EM} using the Itô--Taylor expansion in~\eqref{eq:EM_scheme_with_RM}, but this can be partly avoided by rewriting Itô's lemma in~\eqref{eq:ItoTaylorF} in terms of the multi-indexed functions in~\eqref{eq:DefMultIndCoefFunc} and in terms of the multi-indexed integrals~\eqref{eq:DefMultIndInt}. To this end, we insert $F_{\alpha}$ into~\eqref{eq:ItoTaylorF} to obtain
\begin{align*}
F_{\alpha}(Z(s)) &= F_{\alpha}(Z(s_{0})) + \int_{s_{0}}^{s} \Lo[F_{\alpha}](Z(s_{1})) \diff s_{1} + \int_{s_{0}}^{s} \Li[F_{\alpha}](Z(s_{1})) \diff B(s_{1}) \\ &= F_{\alpha}(Z(s_{0})) + \int_{s_{0}}^{s} F_{(0,\alpha)}(Z(s_{1})) \diff s_{1} + \int_{s_{0}}^{s} F_{(1,\alpha)}(Z(s_{1})) \diff B(s_{1}).
\end{align*}
Next, applying $I_{\alpha}[.]_{s_{0},s}$ to both side of the above yields
\begin{equation}\label{eq:IaF}
I_{\alpha}[F_{\alpha}(Z)]_{s_{0},s} = F_{\alpha}(Z(s_{0})) I_{\alpha}(s_{0},s) + I_{(0,\alpha)}[F_{(0,\alpha)}(Z)]_{s_{0},s} + I_{(1,\alpha)}[F_{(1,\alpha)}(Z)]_{s_{0},s},
\end{equation}
which describes how the order of the expansion is increased using only the compact multi-index notation. Let us now introduce another hierarchical set
\begin{equation}\label{eq:index_A_Milstein}
A_{1} = \{ \emptyset,\ (0),\ (1),\ (1,1) \},
\end{equation}
and the corresponding boundary set
\begin{equation}\label{eq:index_B_Milstein}
B(A_{1}) = \{ \alpha \not\in A_{1}:\ -\alpha \in A_{1} \} = \{ (0,0),\ (0,1),\ (0,1,1),\ (1,1,1) \}.
\end{equation}
The set $A_{1}$ is the set of multi-indices that correspond to the classical Milstein scheme~\cite{MR356225}. The only difference between $A_{0.5}$ in~\eqref{eq:index_A_EM} and $A_{1}$ in~\eqref{eq:index_A_Milstein} is the multi-index $(1,1)$. Thus, to obtain the analogous expression as~\eqref{eq:IT_EM} with $A_{1}$ and $B(A_{1})$, we use~\eqref{eq:IaF} with $\alpha = (1,1)$
\begin{equation*}
I_{(1,1)}[F_{(1,1)}]_{s_{0},s} = F_{(1,1)}(X(s_{0})) I_{(1,1)}(s_{0},s) + I_{(0,1,1)} [F_{(0,1,1)}(X)]_{s_{0},s} + I_{(1,1,1)}[F_{(1,1,1)}(X)]_{s_{0},s}.
\end{equation*}
By inserting this into~\eqref{eq:IT_F_EM} and comparing with~\eqref{eq:index_A_Milstein} and~\eqref{eq:index_B_Milstein}, we obtain
\begin{equation}\label{eq:IT_Milstein}
F(Z(s)) = \sum_{\alpha \in A_{1}} F_{\alpha}(Z(s_{0})) I_{\alpha}(s_{0},s) + \sum_{\alpha \in B(A_{1})} I_{\alpha}[F_{\alpha}(Z)]_{s_{0},s},\ 0 \leq s_{0} \leq s \leq T,
\end{equation}
which is the Itô--Taylor expansion of order $1$. 

The Euler--Maruyama set $A_{0.5}$, corresponding to strong convergence of order $0.5$, and the Milstein set $A_{1}$, corresponding to strong convergence of order $1$, can be generalised to $A_{\gamma}$ for $2 \gamma \in \mathbb{N}$. To this end, for a multi-index $\alpha = (\alpha_{1},\ldots,\alpha_{L}) \in \{0,1 \}^{L}$, we let $\ell(\alpha) = L$ denote the length of $\alpha$ and $n(\alpha)$ denote the number of zeros of $\alpha$. We define
\begin{equation}\label{eq:def_Agamma}
A_{\gamma} = \{ \alpha:\ \ell(\alpha) + n(\alpha) \leq 2 \gamma \text{ or } \ell(\alpha) = n(\alpha) = \gamma + 1/2 \},
\end{equation}
for $2 \gamma \in \mathbb{N}$, and with corresponding remainder set
\begin{equation*}
B(A_{\gamma}) = \{ \alpha \not\in A_{\gamma}:\ -\alpha \in A_{\gamma} \}.
\end{equation*}
Note that this coincides with the previous definitions of $A_{0,5}, B(A_{0.5}), A_{1}$, and $B(A_{1})$, and that $|A_{\gamma}|, |B(A_{\gamma})| < \infty$ for all $2 \gamma \in \mathbb{N}$. Similarly to~\eqref{eq:IT_F_EM} and~\eqref{eq:IT_Milstein}, the Itô--Taylor expansion of $F(Z(s))$ of order $\gamma$ around $s_{0}$ is then given by
\begin{equation}\label{eq:IT_gen}
F(Z(s)) = \sum_{\alpha \in A_{\gamma}} F_{\alpha}(Z(s_{0})) I_{\alpha}(s_{0},s) + \sum_{\alpha \in B(A_{\gamma})} I_{\alpha}[F_{\alpha}(Z)]_{s_{0},s},\ 0 \leq s_{0} \leq s \leq T.
\end{equation}
We now define the time-continuous Itô--Taylor approximation of order $\gamma$. To this end, we insert $F = \id$, with $\id(r) = r$ for $r \in \mathbb{R}$, in~\eqref{eq:IT_gen}, and use the equivalent formulation for the discrete time grid $0=t_{0} < \ldots < t_{M} = T$, to obtain
\begin{equation}\label{eq:IT_gen_v2}
Z(s) = Z(t_{m}) + \sum_{\alpha \in A_{\gamma} \setminus \{  \emptyset \}} \id_{\alpha}(Z(t_{m})) I_{\alpha}(t_{m},s) + \sum_{\alpha \in B(A_{\gamma})} I_{\alpha}[\id_{\alpha}(Z)]_{t_{m},s},\ s \in [t_{m},t_{m+1}],
\end{equation}
where we excluded the term $\id_{\emptyset}(Z(t_{m}))$ (corresponding to $\alpha = \emptyset$) from the first sum to make it more similar to typical numerical schemes. We discard the remainder term in~\eqref{eq:IT_gen_v2} to obtain the time-continuous recursive Itô--Taylor approximation of strong order $\gamma \in \{0.5,1,1.5,\ldots \}$
\begin{equation}\label{eq:IT_approx}
Z^{\IT \gamma}(s) = Z^{\IT \gamma}(t_{m}) + \sum_{\alpha \in A_{\gamma} \setminus \{  \emptyset \}} \id_{\alpha}(Z^{\IT \gamma}(t_{m})) I_{\alpha}(t_{m},s),\ s \in [t_{m},t_{m+1}],
\end{equation}
initialised with $Z^{\IT \gamma}(0) = z_{0}$. The following proposition quantifies the introduced error by discarding the remainder term in the Itô--Taylor expansion in~\eqref{eq:IT_gen_v2}.
\begin{proposition}\label{prop:ITexpConv}
Let $\gamma \in \{0.5,1,1.5,\ldots \}$ and suppose that
\begin{equation*}
\left| F_{\alpha}(r_{1}) - F_{\alpha}(r_{2}) \right| \leq \Lip_{F_{\alpha}} \left| r_{1}-r_{2} \right|,\ \forall r_{1},r_{2} \in \mathbb{R},
\end{equation*}
for all $\alpha \in A_{\gamma}$ and that 
\begin{equation*}
\left| F_{\alpha}(r) \right| \leq L_{F_{\alpha}} \left( 1 + \left| r \right| \right),\ \forall r \in \mathbb{R},
\end{equation*}
for all $\alpha \in B(A_{\gamma})$. Then, for every $p \in (0,\infty)$, it holds
\begin{equation*}
\left( \E \left[ \sup_{s \in [0,T]} \left| Z^{\IT \gamma}(s) - Z(s) \right|^{p} \right] \right)^{1/p} \leq C \Delta t^{\gamma},
\end{equation*}
where the constant $C>0$ is independent of $\Delta t$.
\end{proposition}
The proof of Proposition~\ref{prop:ITexpConv} follows from modifying the proof of Theorem $1$ in \cite{MR1459485} using Lemma $4.1$ in \cite{MR1417823}. As we aim to use~\eqref{eq:IT_approx} to approximate the solution of the SDE~\eqref{eq:LITsdeAdd} with additive noise, we next discuss how~\eqref{eq:IT_approx} simplifies when the diffusion coefficient $\sigma \equiv 1$.

\subsection{Itô--Taylor schemes for SDEs with additive noise}\label{sec:IT_add_noise}
In this section, we describe how the general Itô--Taylor approximation in~\eqref{eq:IT_approx} simplifies when the diffusion coefficient function $\sigma \equiv 1$. We return to the specific notation introduced in Section~\ref{sec:setting} (including using $t$ for the time variable). More precisely, we consider SDEs of the form
\begin{equation}\label{eq:IT_SDE_additive}
\left\lbrace
\begin{aligned}
& \diff Y(t) = H(Y(t)) \diff t + \diff B(t),\ t \in (0,T], \\
& Y(0) = \Phi(x_{0});
\end{aligned}
\right.
\end{equation}
that is, the drift coefficient function $\mu = H$ with $H$ defined in~\eqref{eq:Lamp_drift}.

The Itô--Taylor scheme of order $\gamma \in \{0.5,1,1.5,2,\ldots \}$ (abbreviated by $\IT \gamma$) from $t = t_{m}$ to $t=t_{m+1}$ applied to~\eqref{eq:IT_SDE_additive} is given by~\eqref{eq:IT_approx} with $s = t_{m+1}$. More precisely, we initialise $Y^{\IT \gamma}_{0} = \Phi(x_{0})$, and we recursively define
\begin{equation}\label{eq:IT_approx_add}
Y^{\IT \gamma}_{m+1} = Y^{\IT \gamma}_{m} +  \sum_{\alpha \in A_{\gamma} \setminus \{  \emptyset \}} \operatorname{id}_{\alpha}(Y^{\IT \gamma}_{m}) I_{\alpha}(t_{m},t_{m+1}),\ m=0,\ldots,M-1.
\end{equation}
We remark that the operators $\Lo$ and $\Li$ in~\eqref{eq:LoDef} and in~\eqref{eq:LiDef}, respectively, are now with respect to $\mu = H$ and $\sigma \equiv 1$. In the specific setting of the SDE in~\eqref{eq:IT_SDE_additive}, we prove that the regularity parameter $k \in \{2,3,4,\ldots\}$ governing the regularity of the coefficient functions in~\eqref{eq:SDE} determines precisely which $\IT \gamma$ schemes are convergent. The key ingredient in the proof of strong convergence of the $\IT \gamma$ schemes is the following lemma combined with Proposition~\ref{prop:ITexpConv}.

\begin{lemma}\label{lem:highDerLemma}
Let $\alpha \in \{ 0,1 \}^{L}$ be a multi-index. Then the highest derivative of $H$ in $\id_{\alpha}$ is bounded from above by $\ell(\alpha) + n(\alpha) - 2$. Moreover, the following holds:
\begin{enumerate}
\item If $\alpha \in A_{\gamma}$, then
\begin{equation*}
\ell(\alpha) + n(\alpha) - 2 \leq \begin{cases} 2 \gamma - 2,\ \text{for } \gamma = m \in \{1,2,\ldots \},\\ 2 \gamma - 1,\ \text{for } \gamma = m + 1/2,\ m \in \{0,1,2,\ldots \}. \end{cases}
\end{equation*}
\item If $\alpha \in B(A_{\gamma})$, then 
\begin{equation*}
\ell(\alpha) + n(\alpha) - 2 \leq \begin{cases} 2 \gamma,\ \text{for } \gamma = m \in \{1,2,\ldots \},\\ 2 \gamma + 1,\ \text{for } \gamma = m + 1/2,\ m \in \{0,1,2,\ldots \}. \end{cases}
\end{equation*}
\end{enumerate}
\end{lemma}
\begin{proof}
If $\alpha = (\alpha_{1},\ldots,\alpha_{L-1},1)$ ends with a $1$ and is of length at least $2$, then
\begin{equation*}
\id_{\alpha} = \LL^{\alpha_{1}} \dots \LL^{\alpha_{L-1}} \Li [\id] \equiv 0
\end{equation*}
Indeed, $\Li [\id] = 1$, and therefore $\LL^{\alpha_{L-1}} \Li [\id] \equiv 0$. Thus, if the last index $\alpha_{L} = 1$ then $\id_{\alpha} \equiv 0$, and the statement follows immediately. We assume in the rest of the proof that $\alpha = (\alpha_{1},\ldots,\alpha_{L-1},0)$.

From the general formulas for $\Lo$ and $\Li$, any $0$ in a multi-index $\alpha \in \{ 0,1 \}^{L}$ raises the derivative order by at most $2$ and any $1$ in a multi-index $\alpha$ raises the derivative order by at most $1$. Thus, the highest possible derivative order of $\id_{\alpha}$ for $\alpha = (\alpha_{1},\ldots,\alpha_{L-1},0)$ is
\begin{align*}
2 \left| \{ \text{zeros in } (\alpha_{1},\ldots,\alpha_{L-1}) \} \right| + \left| \{ \text{ones in } (\alpha_{1},\ldots,\alpha_{L-1}) \} \right| &= 2 n (\alpha -) + \ell(\alpha-) - n(\alpha-) \\ &= \ell(\alpha) + n(\alpha) - 2,
\end{align*}
where we used that
\begin{equation*}
\ell(\alpha-)= \ell(\alpha)-1,\ n(\alpha-) = n(\alpha)-1,
\end{equation*}
and that
\begin{equation*}
\id_{\alpha} = \LL^{j_{1}} \ldots \LL^{j_{L-1}} [H].
\end{equation*}

Suppose now that $\alpha \in A_{\gamma}$. If $\gamma = m \in \{1,2,3,\ldots \}$ is an integer, then by the definition of $A_{\gamma}$, we have that
\begin{equation*}
\ell(\alpha) + n(\alpha) \leq 2 \gamma,
\end{equation*}
which implies that
\begin{equation*}
\ell(\alpha) + n(\alpha) - 2 \leq 2 \gamma -2.
\end{equation*}
If $\gamma = m + 1/2$, with $m \in \{0,1,2,\ldots \}$, then by the definition of $A_{\gamma}$, we have two cases: $1)$ $\ell(\alpha) + n(\alpha) \leq 2 \gamma$ or $2)$ $\ell(\alpha) = n(\alpha) = \gamma + 1/2$. Case $1)$ implies, as above, that
\begin{equation*}
\ell(\alpha) + n(\alpha) - 2 \leq 2 \gamma -2
\end{equation*}
and case $2)$ implies that
\begin{equation*}
\ell(\alpha) + n(\alpha) -2 = 2 \gamma -1.
\end{equation*}
Thus, in summary, if $\gamma = m + 1/2$, with $m \in \{ 0,1,2,\ldots \}$, then $\ell(\alpha) + n(\alpha) - 2 \leq 2 \gamma - 1$.

Suppose now that $\alpha \in B(A_{\gamma})$. By definition of $B(A_{\gamma})$, if $\alpha \in B(A_{\gamma})$ then $-\alpha \in A_{\gamma}$. Since $-\alpha \in A_{\gamma}$, we can apply the above to $-\alpha$
\begin{equation*}
\ell(-\alpha) + n(-\alpha) - 2 \leq \begin{cases} 2 \gamma - 2,\ \text{for } \gamma = m \in \{1,2,\ldots \},\\ 2 \gamma - 1,\ \text{for } \gamma = m + 1/2,\ m \in \{0,1,2,\ldots \}. \end{cases}
\end{equation*}
By definition,
\begin{equation*}
\ell(\alpha) = \ell(-\alpha) - 1,
\end{equation*}
and we can estimate
\begin{equation*}
n(\alpha) \leq n(-\alpha)+1.
\end{equation*}
The latter uses that
\begin{equation*}
n(\alpha) = n(-\alpha)
\end{equation*}
for $\alpha_{1}=1$ and that
\begin{equation*}
n(\alpha) = n(-\alpha) +1
\end{equation*}
for $\alpha_{1} = 0$. This gives us the desired estimate
\begin{equation*}
\ell(\alpha) + n(\alpha) - 2 \leq \ell(-\alpha) + n(-\alpha) - 2 + 2 \leq \begin{cases} 2 \gamma,\ \text{for } \gamma = m \in \{1,2,\ldots \},\\ 2 \gamma + 1,\ \text{for } \gamma = m + 1/2,\ m \in \{0,1,2,\ldots \}. \end{cases}
\end{equation*}
\end{proof}

We next state and prove strong convergence of the $\IT \gamma$ scheme. This implies that the $\IT \gamma$ schemes satisfy Assumption~\ref{ass:globalStrongInt} under the regularity assumptions in Section~\ref{sec:setting}.
\begin{proposition}\label{prop:IT_conv_add}
Suppose that Assumptions~\ref{ass:f},~\ref{ass:g}, and~\ref{ass:fg} are satisfied with $k = 2 \gamma + 1$ for $2 \gamma \in \{1,3,5,\ldots \}$. Then the $\IT \gamma$ scheme satisfies
\begin{equation*}
\PP (Y^{\IT \gamma}_{m} \in \mathbb{R},\ \forall m = 0,\ldots,M) = 1,
\end{equation*}
and, for every $p \in (0,\infty)$, is $p$-strongly convergent with order $\gamma$:
\begin{equation*}
\left( \E \left[ \sup_{m=0, \ldots, M} | Y_{m}^{\IT \gamma} - Y(t_{m}) |^{p} \right] \right)^{\frac{1}{p}} \leq C \Delta t^{\gamma},
\end{equation*}
where the constant $C>0$ is independent of $\Delta t >0$. Similarly, the $\IT (\gamma+1/2)$ scheme satisfies
\begin{equation*}
\PP (Y^{\IT (\gamma+1/2)}_{m} \in \mathbb{R},\ \forall m = 0,\ldots,M) = 1,
\end{equation*}
and, for every $p \in (0,\infty)$, is $p$-strongly convergent with order $\gamma+1/2$:
\begin{equation*}
\left( \E \left[ \sup_{m=0, \ldots, M} | Y_{m}^{\IT (\gamma+1/2)} - Y(t_{m}) |^{p} \right] \right)^{\frac{1}{p}} \leq C \Delta t^{\gamma+1/2},
\end{equation*}
where the constant $C>0$ is independent of $\Delta t >0$.
\end{proposition}
\begin{proof}
The statements follow if we prove that $\id_{\alpha}$ is Lipschitz continuous and bounded for every $\alpha \in A_{\gamma}$ and for every $\alpha \in A_{\gamma+1/2}$, respectively, and that $\id_{\alpha}$ is of linear growth for every $\alpha \in B \left( A_{\gamma} \right)$ and for every $\alpha \in B\left( A_{\gamma+1/2} \right)$, respectively. Firstly, if $\id_{\alpha}$ is bounded for every $\alpha \in A_{\gamma}$ and if $Y^{\IT \gamma}_{m} \in \mathbb{R}$ then
\begin{equation*}
Y^{\IT \gamma}_{m+1} = Y^{\IT \gamma}_{m} +  \sum_{\alpha \in A_{\gamma} \setminus \{  \emptyset \}} \operatorname{id}_{\alpha}(Y^{\IT \gamma}_{m}) I_{\alpha}(t_{m},t_{m+1}) \in \mathbb{R},
\end{equation*}
where we also used Lemma~\ref{lem:Ialpha_bounded} over finitely many $\alpha \in A_{\gamma} \setminus \{  \emptyset \}$. By induction over $m=0,\ldots,M$, we obtain
\begin{equation*}
\PP (Y^{\IT \gamma}_{m} \in \mathbb{R},\ \forall m = 0,\ldots,M) = 1.
\end{equation*}
Secondly, if $\id_{\alpha}$ is Lipschitz continuous for every $\alpha \in A_{\gamma}$ and if $\id_{\alpha}$ is of linear growth for every $\alpha \in B \left( A_{\gamma} \right)$ then the convergence statement for the $\IT \gamma$ scheme follows from Proposition~\ref{prop:ITexpConv}. The same arguments give us the corresponding statements for the $\IT (\gamma+1/2)$ scheme. 

Let us first consider the $\IT \gamma$ scheme. By Lemma~\ref{lem:highDerLemma}, the highest derivative of $H$ in $\id_{\alpha}$ for $\alpha \in A_{\gamma}$ is
\begin{equation*}
\ell(\alpha) + n(\alpha)-2 \leq 2 \gamma - 1,
\end{equation*}
since $\gamma \in \{0.5,1.5,2.5,\ldots \}$. Since $H \in \C^{k}_{b}(\mathbb{R})$ and $k=2\gamma+1$, we conclude that $\id_{\alpha}$ is Lipschitz continuous and bounded for all $\alpha \in A_{\gamma}$. Similarly, by Lemma~\ref{lem:highDerLemma}, the highest derivative of $H$ in $\id_{\alpha}$ for $\alpha \in B(A_{\gamma})$ is
\begin{equation*}
\ell(\alpha) + n(\alpha) - 2 \leq 2 \gamma + 1,
\end{equation*}
since $\gamma \in \{ 0.5,1.5,2.5,\ldots \}$. We conclude that $\id_{\alpha}$ is of linear growth for all $\alpha \in B(A_{\gamma})$, since $H \in \C^{k}_{b}(\mathbb{R})$ and $k=2\gamma+1$. This implies, by Proposition~\ref{prop:ITexpConv}, that the $\IT \gamma$ scheme is p-strongly convergent with order $\gamma$.

Let us now consider the $\IT (\gamma+1/2)$ scheme. The argument is the same as the argument for the $\IT \gamma$ scheme. If $\alpha \in A_{\gamma + 1/2}$, then we use Lemma~\ref{lem:highDerLemma} to bound
\begin{equation*}
\ell(\alpha) + n(\alpha) - 2 \leq 2 (\gamma + 1/2) - 2 = 2 \gamma - 1.
\end{equation*}
Since $H \in \C^{2 \gamma + 1}_{b}(\mathbb{R})$, we conclude that $\id_{\alpha}$ is Lipschitz continuous and bounded for all $\alpha \in A_{\gamma + 1/2}$. If $\alpha \in B(A_{\gamma + 1/2})$, then we can similarly bound
\begin{equation*}
\ell(\alpha) + n(\alpha) - 2 \leq 2 (\gamma + 1/2) = 2 \gamma + 1,
\end{equation*}
again using Lemma~\ref{lem:highDerLemma}. We conclude that $\id_{\alpha}$ is of linear growth for all $\alpha \in B(A_{\gamma})$, since $H \in \C^{k}_{b}(\mathbb{R})$ and $k=2 \gamma + 1$. This completes the proof.
\end{proof}
We remark that Proposition~\ref{prop:IT_conv_add} is also true for the time-continuous analogue of the time-discrete scheme $Y^{\IT \gamma}_{m}$ in~\eqref{eq:IT_approx_add}. We use the discrete time definition as this is what is implementable.

\subsection{Lamperti--Itô--Taylor schemes}\label{sec:LIT_schemes}
We next provide the definition of the Lamperti--Itô--Taylor-$\gamma$ ($\LIT \gamma$) scheme of strong order $\gamma$, for $\gamma \in \{0.5,1,1.5,2,\ldots \}$, for the considered SDE
\begin{equation*}
\left\lbrace
\begin{aligned}
& \diff X(t) = f(X(t)) \diff t + g(X(t)) \diff B(t),\ t \in (0,T], \\ 
& X(0) = x_{0} \in \D.
\end{aligned}
\right.
\end{equation*}
We define the $\LIT \gamma$ scheme as
\begin{equation}\label{eq:LIT_def}
X^{\LIT \gamma}_{m} = \Phi^{-1}(Y^{\IT \gamma}_{m}),\ m=0,\ldots,M,
\end{equation}
where $Y^{\IT \gamma}_{m}$ is the one-step $\IT \gamma$ scheme defined in~\eqref{eq:IT_approx_add}.

The convergence of the $\LIT \gamma$ scheme defined~\eqref{eq:LIT_def} follows from Proposition~\ref{prop:IT_conv_add} and Theorem~\ref{theo:mainStrong}.

\begin{corollary}\label{cor:LIT_conv}
Let $M\in\N$, $T > 0$, $\Delta t = T/M$ and let $x_{0} \in \D$. Suppose Assumptions~\ref{ass:f},~\ref{ass:g},~\ref{ass:fg}, and ~\ref{ass:globalStrongInt} are satisfied with $k = 2 \gamma + 1$ for $2 \gamma \in \{1,3,5,\ldots \}$. Then the $\LIT \gamma$ scheme satisfies
\begin{equation*}
\PP (X^{\LIT \gamma}_{m} \in \D,\ \forall m = 0,\ldots,M) = 1,
\end{equation*}
and, for every $p \in (0,\infty)$, is $p$-strongly convergent with order $\gamma$:
\begin{equation*}
\left( \E \left[ \sup_{m=0, \ldots, M} | X_{m}^{\LIT \gamma} - X(t_{m}) |^{p} \right] \right)^{\frac{1}{p}} \leq C \Delta t^{\gamma},
\end{equation*}
where the constant $C>0$ is independent of $\Delta t >0$. Similarly, the $\LIT (\gamma+1/2)$ scheme satisfies
\begin{equation*}
\PP (X^{\LIT (\gamma+1/2)}_{m} \in \D,\ \forall m = 0,\ldots,M) = 1,
\end{equation*}
and, for every $p \in (0,\infty)$, is $p$-strongly convergent with order $\gamma+1/2$:
\begin{equation*}
\left( \E \left[ \sup_{m=0, \ldots, M} | X_{m}^{\LIT (\gamma+1/2)} - X(t_{m}) |^{p} \right] \right)^{\frac{1}{p}} \leq C \Delta t^{\gamma+1/2},
\end{equation*}
where the constant $C>0$ is independent of $\Delta t >0$.
\end{corollary}
\begin{proof}
Boundary preservation and the convergence statement follow immediately from Theorem~\ref{theo:mainStrong} and Proposition~\ref{prop:IT_conv_add}.
\end{proof}
Corollary~\ref{cor:LIT_conv} tell us that, for a given regularity parameter $k = 2 \gamma + 1$, with $2 \gamma \in \{1,3,5,\ldots \}$, both $\LIT \gamma$ and $\LIT (\gamma+1/2)$ schemes converge with their corresponding order. Note that Corollary~\ref{cor:LIT_conv} is also true for the time-continuous analogue of $X^{\LIT \gamma}$. We use the time-discrete version since this is what is implementable. Next, we provide the explicit recursive formulas for the $\LIT \gamma$ schemes used for the numerical experiments in Section~\ref{sec:numExp}.

\subsubsection{The Lamperti--EM scheme}
The Itô--Taylor-$0.5$ scheme is also known as the Euler--Maruyama (EM) scheme, and we use the term EM as this is more common in the literature. We recall that
\begin{equation*}
A_{0.5} = \{ \emptyset, (0), (1) \}
\end{equation*}
and we insert
\begin{equation*}
I_{(0)}(t_{m},t_{m+1}) = t_{m+1} - t_{m}
\end{equation*}
and
\begin{equation*}
I_{(1)}(t_{m},t_{m+1}) = B(t_{m+1}) - B(t_{m}) = \Delta B_{m},
\end{equation*}
into~\eqref{eq:IT_approx_add} to obtain the EM scheme given by
\begin{equation*}
Y^{\EM}_{m+1} = Y^{\EM}_{m} + H(Y^{\EM}_{m}) \Delta t + \Delta B_{m},\ m=0,\ldots,M-1,
\end{equation*}
with $\Delta t = T/M$, initialised with $Y^{\EM}_{0} = \Phi(x_{0})$. We define the Lamperti--EM (LEM) scheme by
\begin{equation*}
X^{\LEM}_{m} = \Phi^{-1}(Y^{\EM}_{m}),\ m=0,\ldots,M.
\end{equation*}
By Corollary~\ref{cor:LIT_conv}, if $k \in \{2,4,6,\ldots \}$ then, for every $p \in (0,\infty)$, the $\LEM$ scheme is p-strongly convergent with order $0.5$. In fact, as is shown in Section~\ref{sec:Lamperti_Milstein}, the $\LEM$ scheme coincides in this case with the $\LIT 1.0$ scheme. Therefore, by Corollary~\ref{cor:LIT_conv}, if $k \in \{2,4,6,\ldots \}$ then, for every $p \in (0,\infty)$, the $\LEM$ scheme is p-strongly convergent with order $1.0$.

\subsubsection{The Lamperti--Itô--Taylor-$1$ scheme}\label{sec:Lamperti_Milstein}
The Itô--Taylor-$1$ scheme is also known as the Milstein scheme, and we refer to the scheme as the Milstein scheme, as this is more common in the literature. We recall that 
\begin{equation*}
A_{1} = \{ \emptyset,\ (0),\ (1),\ (1,1) \}
\end{equation*}
and we insert
\begin{equation*}
I_{(0)}(t_{m},t_{m+1}) = t_{m+1} - t_{m}
\end{equation*}
and
\begin{equation*}
I_{(1)}(t_{m},t_{m+1}) = B(t_{m+1}) - B(t_{m}) = \Delta B_{m},
\end{equation*}
into~\eqref{eq:IT_approx_add} to obtain the Milstein scheme
\begin{equation*}
Y^{\MIL}_{m+1} = Y^{\MIL}_{m} + H(Y^{\MIL}_{m}) \Delta t + \Delta B_{m} + \id_{(1,1)}(Y^{\MIL}_{m}) I_{(1,1)}(t_{m},t_{m+1}),\ m=0,\ldots,M-1.
\end{equation*}
In other words, the additional term compared to the EM scheme is the one corresponding to the multi-index $(1,1)$. Using~\eqref{eq:DefMultIndCoefFunc}, we obtain that it vanishes
\begin{equation*}
\id_{(1,1)}(Y^{\MIL}_{m}) = 0.
\end{equation*}
This is the well-known result that, in the case with additive noise and provided that the Milstein scheme is well-defined, the EM and Milstein schemes coincide. Therefore, by Corollary~\ref{cor:LIT_conv}, the EM scheme converges $p$-strongly with order $1$, for every $p \in (0,\infty)$, in this case, provided that $k \in \{2,4,6,\ldots \}$.

\subsubsection{The Lamperti--Itô--Taylor-$1.5$ scheme}
By using the definition of $A_{\gamma}$ in~\eqref{eq:def_Agamma} for $\gamma=1.5$, we see that
\begin{equation*}
A_{1.5} = \{ \emptyset,\ (0),\ (1),\ (1,1),\ (0,1),\ (1,0),\ (0,0),\ (1,1,1) \}.
\end{equation*}
In other words, the additional terms compared to the Milstein scheme are the four terms corresponding to the multi-indices $(0,1),\ (1,0),\ (0,0),\ (1,1,1)$. Using~\eqref{eq:DefMultIndCoefFunc}, we compute
\begin{equation*}
\id_{(0,1)}(Y(t_{m})) = 0,
\end{equation*}
\begin{equation*}
\id_{(1,0)}(Y(t_{m})) = H'(Y(t_{m})),
\end{equation*}
\begin{equation*}
\id_{(0,0)}(Y(t_{m})) = H(Y(t_{m})) H'(Y(t_{m})) + \frac{1}{2} H''(Y(t_{m})),
\end{equation*}
and
\begin{equation*}
\id_{(1,1,1)}(Y(t_{0})) = 0.
\end{equation*}
Combining this with~\eqref{eq:DefMultIndInt}, the additional terms compared to the Milstein scheme are
\begin{equation*}
\id_{(1,0)}(Y(t_{m})) I_{(1,0)}(t_{m},t_{m+1}) =  H'(Y(t_{m})) I_{(1,0)}(t_{m},t_{m+1}),
\end{equation*}
and
\begin{equation*}
\id_{(0,0)}(Y(t_{m})) I_{(0,0)}(t_{m},t_{m+1}) = \left( H(Y(t_{m})) H'(Y(t_{m})) + \frac{1}{2} H''(Y(t_{m})) \right) \frac{(t_{m+1}-t_{m})^{2}}{2}.
\end{equation*}
For numerical experiments involving Itô--Taylor schemes of order strictly higher than $1/2$, we need more information about the Brownian motion path than what is provided by the typical Brownian motion increments $\Delta B_{m}$ for $m=0,\ldots,M-1$. In our case, this means that we sample $I_{(1,0)}(t_{m},t_{m+1})$, coupled with the Brownian motion path, exactly (see Section~\ref{sec:comp_iter_ints}). Therefore, we let this term remain in the following formulas.

Inserting the above into~\eqref{eq:IT_approx_add} with $\gamma = 1.5$ gives us the Itô--Taylor-$1.5$ ($\IT 1.5$) scheme given by
\begin{equation*}\label{eq:IT_15_add}
  \begin{split}
    Y^{\IT 1.5}_{m+1} &= Y^{\IT 1.5}_{m} + H(Y^{\IT 1.5}_{m}) \Delta t + \Delta B_{m} \\ &+ H'(Y^{\IT 1.5}_{m}) I_{(1,0)}(t_{m},t_{m+1}) \\ &+ \left( H(Y^{\IT 1.5}_{m}) H'(Y^{\IT 1.5}_{m}) + \frac{1}{2} H''(Y^{\IT 1.5}_{m}) \right) \frac{\Delta t^{2}}{2},\ m=0,\ldots,M-1,
  \end{split}
\end{equation*}
with $\Delta t = T/M$, initialised with $Y^{\IT 1.5}_{0} = \Phi(x_{0})$. We define the Lamperti--Itô--Taylor-$1.5$ ($\LIT 1.5$) scheme by
\begin{equation}\label{eq:LIT_15_add}
X^{\LIT 1.5}_{m} = \Phi^{-1}(Y^{\IT 1.5}_{m}),\ m=0,\ldots,M.
\end{equation}
By Corollary~\ref{cor:LIT_conv}, if $k \in \{4,6,8, \ldots \}$ then, for every $p \in (0,\infty)$, the $\LIT 1.5$ scheme is p-strongly convergent with order $1.5$.

\subsubsection{The Lamperti--Itô--Taylor-$2$ scheme}
By using the definition of $A_{\gamma}$ in~\eqref{eq:def_Agamma} with $\gamma=2$, we see that
\begin{equation*}
A_{2} = \{ \emptyset,\ (0),\ (1),\ (1,1),\ (0,1),\ (1,0),\ (0,0),\ (1,1,1),\ (1,1,0),\ (1,0,1),\ (0,1,1),\ (1,1,1,1) \}.
\end{equation*}
In other words, the additional terms compared to the Itô--Taylor-$1.5$ scheme are the four terms corresponding to the multi-indices $(1,1,0),\ (1,0,1),\ (0,1,1),\ (1,1,1,1)$. We compute
\begin{equation*}
\id_{(1,1,0)}(Y(t_{m})) = H''(Y(t_{m})),
\end{equation*}
\begin{equation*}
\id_{(1,0,1)}(Y(t_{m})) = 0,
\end{equation*}
\begin{equation*}
\id_{(0,1,1)}(Y(t_{m})) = 0,
\end{equation*}
and
\begin{equation*}
\id_{(1,1,1,1)}(Y(t_{m})) = 0.
\end{equation*}
Combining the above with~\eqref{eq:DefMultIndInt}, the additional term compared to the $\IT 1.5$ scheme is
\begin{equation*}
\id_{(1,1,0)}(X(t_{m})) I_{(1,1,0)}(t_{m},t_{m+1}) = H''(Y(t_{m})) I_{(1,1,0)}(t_{m},t_{m+1}).
\end{equation*}
We describe how to approximate $I_{(1,1,0)}(t_{m},t_{m+1})$ in Section~\ref{sec:comp_iter_ints}. Inserting the above into~\eqref{eq:IT_approx_add} with $\gamma = 2$ gives us the Itô--Taylor-$2$ ($\IT 2.0$) scheme given by
\begin{align*}
Y^{\IT 2.0}_{m+1} &= Y^{\IT 2.0}_{m} + H(Y^{\IT 2.0}_{m}) \Delta t + \Delta B_{m} \\ &+ H'(Y^{\IT 2.0}_{m}) I_{(1,0)}(t_{m},t_{m+1}) \\ &+ \left( H(Y^{\IT 2.0}_{m}) H'(Y^{\IT 2.0}_{m}) + \frac{1}{2} H''(Y^{\IT 2.0}_{m}) \right) \frac{\Delta t^{2}}{2} \\ &+ H''(Y^{\IT 2.0}_{m}) I_{(1,1,0)}(t_{m},t_{m+1}),\ m=0,\ldots,M-1,
\end{align*}
with $\Delta t = T/M$ and $\Delta B_{m} = B(t_{m+1})-B(t_{m})$, initialised with $Y^{\LIT 2.0}_{0} = \Phi(x_{0})$. We define the Lamperti--Itô--Taylor-$2.0$ ($\LIT 2.0$) scheme by
\begin{equation*}
X^{\LIT 2.0}_{m} = \Phi^{-1}(Y^{\IT 2.0}_{m}),\ m=0,\ldots,M.
\end{equation*}
By Corollary~\ref{cor:LIT_conv}, if $k \in \{4,6,8, \ldots \}$ then, for every $p \in (0,\infty)$, the $\LIT 2.0$ scheme is p-strongly convergent with order $2.0$. 

\subsection{Computation of iterated integrals}\label{sec:comp_iter_ints}
The goal of this section is to describe the computation and approximation of
\begin{equation}\label{eq:to_implement}
\Delta B_{m},\ I_{(1,0)}(t_{m},t_{m+1}),\ I_{(1,1,0)}(t_{m},t_{m+1}),\ m=0,\ldots,M-1,
\end{equation}
on the discrete time grids used in the numerical experiments in Section~\ref{sec:numExp}.

We use three different discrete time grids for each sample for the numerical convergence analysis in Section~\ref{sec:numExp}, and they corresponding to three different time grid parameters
\begin{equation*}
M^{\app}, M^{\refe}, M^{\aux} \in \mathbb{N}
\end{equation*}  
with $M^{\app} < M^{\refe} < M^{\aux}$. In the following, we consider one sample, and we thus fix one approximation time grid, one reference time grid, and one auxiliary time grid. The numerical approximation is computed using the approximation time grid $t^{\app}_{m} = m \Delta t^{\app}$, for $m=0,\ldots,M^{\app}$, with $\Delta t^{\app} = T/M^{\app}$. The reference solution is computed using the reference time grid $t^{\refe}_{m} = m \Delta t^{\refe}$, for $m=0,\ldots,M^{\refe}$, with $\Delta t^{\refe} = T/M^{\refe}$. The auxiliary time grid $t^{\aux}_{m} = m \Delta t^{\aux}$, for $m=0,\ldots,M^{\aux}$, with $\Delta t^{\aux} = T/M^{\aux}$, is used to approximate some iterated integrals on the reference time grid. Lastly, we also introduce the quotients
\begin{equation*}
R^{aux} = \frac{M^{aux}}{M^{ref}},\ R^{ref} = \frac{M^{ref}}{M^{approx}},
\end{equation*}
and we assume that $M^{\app} < M^{\refe} < M^{\aux}$ are such that $R^{aux}, R^{ref} \in \{1,2,3,4,\ldots \}$ are integers. Whenever we refer to $t_{m}$, this means that the statement applies to any of the considered discrete time grids.

We remark that, to obtain strong convergence between an approximation and a reference solution, we have to make sure that all quantities in~\eqref{eq:to_implement} are computed using the same sample path of the Brownian motion. In other words, we first compute and approximate~\eqref{eq:to_implement} on the reference grid and then use this to compute~\eqref{eq:to_implement} on the approximation grid.

The following lemma, known as Chen's lemma, will be used in the following.
\begin{lemma}\label{lem:Chen_lemma}
Let $0 \leq s \leq u \leq t \leq T$. Then the following holds
\begin{equation*}
I_{\alpha}(s,t) = \sum_{\ell=0}^{L} I_{(\alpha_{1},\ldots,\alpha_{\ell})}(s,u) \times I_{(\alpha_{\ell+1},\ldots,\alpha_{L})}(u,t),
\end{equation*}
where $L = | \alpha| \in \mathbb{N}$ is the length of the multi-index $\alpha = (\alpha_{1},\ldots,\alpha_{L})$.
\end{lemma}
The original version of the lemma for deterministic path integrals was proved in~\cite{MR73174}. Similarly, Lemma~\ref{lem:Chen_lemma} can be proved by induction over the multi-index length $L \in \mathbb{N}$. We omit it since it is not the focus of this work.  

We first generate $\Delta B$ and $I_{(1,0)}$ on the auxiliary time grid, and then compute $\Delta B$ and $I_{(1,0)}$ on the reference and approximate time grids based on the generated values on the auxiliary time grid. Thereafter, we describe how to approximate $I_{(1,1,0)}$ on the reference and approximate discrete time grids in a consistent manner with respect to the Brownian motion sample path.

Since
\begin{equation*}
I_{(1,0)}(s,t) = \int_{s}^{t} I_{(1)}(s,s_{1}) \diff s_{1} = \int_{s}^{t} B(s_{1})-B(s) \diff s_{1}
\end{equation*}
depends on the sample path of the Brownian motion and cannot be simplified further, we have to generate $I_{(1,0)}$ correlated with $\Delta B$ on the auxiliary time grid. We let
\begin{equation*}
Z_{1,m} \sim N(0,1),\ m=0,\ldots, M^{aux}-1
\end{equation*}
and
\begin{equation*}
Z_{2,m} \sim N(0,1),\ m=0,\ldots,M^{aux}-1 
\end{equation*}
be independent standard normal random variables. Then, in distribution, we have that
\begin{equation*}
\Delta B^{aux}_{m} = B(t_{m+1}^{aux}) - B(t_{m}^{aux}) = \sqrt{\Delta t^{aux}} Z_{1,m},\ m=0,\ldots,M^{aux}-1
\end{equation*}
and
\begin{equation*}
I_{(1,0)}(t_{m}^{aux},t_{m+1}^{aux}) = \frac{1}{2} \Delta t^{aux} \Delta B^{aux}_{m} + \frac{(\Delta t^{aux})^{3/2}}{2 \sqrt{3}} Z_{2,m},\ m=0,\ldots,M^{aux}-1.
\end{equation*}
The above follows from, for each $m=0,\ldots,M^{aux}-1$, verifying that both sides are $2$-dimensional Gaussian random variables with the same mean and covariance matrix. This gives us $\Delta B$ and $I_{(1,0)}$ on the auxiliary time grid $t_{0}^{aux},\ldots,t_{M^{aux}-1}^{aux}$. 

Next, we compute $\Delta B$ and $I_{(1,0)}$ on the reference and approximation time grids based on the corresponding quantities on the auxiliary time grid as
\begin{equation*}
\Delta B^{ref}_{m} = B(t_{m+1}^{ref}) - B(t_{m}^{ref}) = \sum_{j=0}^{R^{aux}-1} \Delta B^{aux}_{m R^{aux}+j},\ m=0,\ldots,M^{ref}-1,
\end{equation*}
\begin{equation*}
\Delta B^{approx}_{m} = B(t_{m+1}^{approx}) - B(t_{m}^{approx}) = \sum_{j=0}^{R^{ref}-1} \Delta B^{ref}_{m R^{ref}+j},\ m=0,\ldots,M^{approx}-1,
\end{equation*}
and
\begin{equation}\label{eq:I10_ref}
  \begin{split}
    I_{(1,0)}(t_{m}^{ref},t_{m+1}^{ref}) = \sum_{j=0}^{R^{aux}-1} &\Delta t^{aux} ( B(t_{m}^{ref} + j \Delta t^{aux})-B(t_{m}^{ref})) \\ &+ I_{(1,0)}(t_{m}^{ref} + j \Delta t^{aux},t_{m}^{ref} + (j+1) \Delta t^{aux}),
  \end{split}
\end{equation}
for $m=0,\ldots,M^{ref}-1$, and
\begin{equation}\label{eq:I10_approx}
  \begin{split}
    I_{(1,0)}(t_{m}^{approx},t_{m+1}^{approx}) = \sum_{j=0}^{R^{ref}-1} &\Delta t^{ref} ( B(t_{m}^{approx} + j \Delta t^{ref})-B(t_{m}^{approx})) \\ &+ I_{(1,0)}(t_{m}^{approx} + j \Delta t^{ref},t_{m}^{approx} + (j+1) \Delta t^{ref}),
  \end{split}
\end{equation}
for $m=0,\ldots,M^{approx}-1$. The formulas given in~\eqref{eq:I10_ref} and in~\eqref{eq:I10_approx} follow from Lemma~\ref{lem:Chen_lemma}. Thus, we can compute $I_{(1,0)}(t_{m}^{ref},t_{m+1}^{ref})$ using~\eqref{eq:I10_ref}, for all $m=0,\ldots,M^{ref}$, from $I_{(1,0)}(t_{m}^{aux},t_{m+1}^{aux})$ and $\Delta B_{m}^{aux}$, both for all $m=0,\ldots,M^{aux}$. Similarly, we can compute $I_{(1,0)}(t_{m}^{approx},t_{m+1}^{approx})$ using~\eqref{eq:I10_approx}, for all $m=0,\ldots,M^{approx}$, from $I_{(1,0)}(t_{m}^{ref},t_{m+1}^{ref})$ and from $\Delta B_{m}^{ref}$, both for all $m=0,\ldots,M^{ref}$.

Let us now consider $I_{(1,1,0)}$. Since the outermost integral in $I_{(1,1,0)}$ is a Lebesgue integral, we use the trapezoidal rule to approximate $I_{(1,1,0)}(t_{m}^{ref},t_{m+1}^{ref})$ as
\begin{equation}\label{eq:I110_approx_ref}
  \begin{split}
    I_{(1,1,0)}(t_{m}^{ref},t_{m+1}^{ref}) &= \int_{t_{m}^{ref}}^{t_{m+1}^{ref}} I_{(1,1)}(t_{m}^{ref},s_{1}) \diff s_{1} \\ &\approx \Delta t^{aux} \left( \frac{1}{2} I_{(1,1)}(t_{m}^{ref},t_{m}^{ref}) + \sum_{j=1}^{R^{aux}-1} I_{(1,1)}(t_{m}^{ref},t_{m}^{ref} + j \Delta t^{aux}) \right. \\  & \left.+ \frac{1}{2} I_{(1,1)}(t_{m}^{ref},t_{m+1}^{ref}) \right),\ m=0,\ldots,M^{ref}-1.
  \end{split}
\end{equation}
The terms involving $I_{(1,1)}$ can be calculated from the Brownian motion as
\begin{equation*}
I_{(1,1)}(s,t) = \frac{1}{2}\left( (B(t)-B(s))^2 - (t-s) \right),\ 0 \leq s \leq t \leq T,
\end{equation*}
which means that~\eqref{eq:I110_approx_ref} provides us with a formula to implement an approximation of $I_{(1,1,0)}(t_{m}^{ref},t_{m+1}^{ref})$. Similarly, the analogous formula holds for $I_{(1,1,0)}(t_{m}^{approx},t_{m+1}^{approx})$
\begin{equation*}
  \begin{split}
    I_{(1,1,0)}(t_{m}^{approx},t_{m+1}^{approx}) &= \int_{t_{m}^{approx}}^{t_{m+1}^{approx}} I_{(1,1)}(t_{m}^{approx},s_{1}) \diff s_{1} \\ &\approx \Delta t^{ref} \left( \frac{1}{2} I_{(1,1)}(t_{m}^{approx},t_{m}^{approx}) + \sum_{j=1}^{R^{ref}-1} I_{(1,1)}(t_{m}^{approx},t_{m}^{approx} + j \Delta t^{ref}) \right. \\  & \left.+ \frac{1}{2} I_{(1,1)}(t_{m}^{approx},t_{m+1}^{approx}) \right),\ m=0,\ldots,M^{approx}-1.
  \end{split}
\end{equation*}
This completes the description of the iterated integrals needed for the numerical experiments in Section~\ref{sec:numExp}.

\section{Numerical experiments}\label{sec:numExp}In this section, we provide numerical experiments to numerically verify that the proposed Lamperti--Itô--Taylor-$\gamma$ ($\LIT \gamma$) schemes are boundary-preserving and $2$-strongly convergent with orders as stated and proved in Corollary~\ref{cor:LIT_conv}. More precisely, we numerically verify that the Lamperti--Euler--Maruyama (LEM) scheme achieves $2$-strong convergence order $1$, the Lamperti--Itô--Taylor-$1.5$ (LIT$1.5$) achieves $2$-strong convergence order $1.5$, and the Lamperti--Itô--Taylor-$2.0$ (LIT$2.0$) scheme achieves $2$-strong convergence order $2$. To this end, recall that $M \in \mathbb{N}$ is the number of subintervals $[t_{m},t_{m+1}]$, $m=0,\ldots,M-1$, each of size $\Delta t = T/M$, in the time discretisation, and we let $\Delta B_{m} = B(t_{m+1}) - B(t_{m})$ be the Brownian motion increment over the interval $[t_{m},t_{m+1}]$.

We here consider the following noise-scaled version of the SDE in~\eqref{eq:SDE}
\begin{equation*}
\left\lbrace
\begin{aligned}
& \diff X(t) = f(X(t)) \diff t + \lambda g(X(t)) \diff B(t),\ t \in (0,T], \\ 
& X(0) = x_{0} \in \D,
\end{aligned}
\right.
\end{equation*}
where $\lambda>0$ is introduced to more easily show lack of boundary preservation for the classical numerical schemes that we compare the $\LIT \gamma$ schemes with. Note that the effect of introducing $\lambda>0$ can also be achieved by running the numerical experiments for a larger $T \in (0,\infty)$. To put this into the framework of Section~\ref{sec:setting}, we let $\Phi(r) = \int_{w_{0}}^{r} \frac{1}{g(w)} \diff w$, for $r \in \D$. Then $Y(t) = \Phi(X(t))$ satisfies
\begin{equation*}
\diff Y(t) = \left( \frac{f(\Phi^{-1}(Y(t)))}{g(\Phi^{-1}(Y(t)))} - \frac{\lambda^{2}}{2} g'(\Phi^{-1}(Y(t))) \right) \diff t + \lambda \diff B(t).
\end{equation*}
Note how the introduced noise-scaling parameter $\lambda>0$ changes~\eqref{eq:SDE-Lamp}.

We either use $T = 0.4$ or $T = 1$, and we either use a fixed value of $x_{0}$ or use $x_{0}$ uniformly distributed on the invariant domain $\D$.

Boundary preservation of the $\LIT \gamma$ schemes is compared to the lack of boundary preservation of the following well-known schemes for SDEs
\begin{itemize}
\item the Euler--Maruyama scheme (denoted by EM), see for instance \cite{MR1214374}
$$
X^{\EM}_{m+1}=X^{\EM}_m+ f(X^{\EM}_{m})\Delta t + \lambda g(X^{\EM}_m) \Delta B_{m},\ m=0,\ldots,M-1,
$$
initialised with $X^{\EM}_{0} = x_{0}$.
\item the semi-implicit Euler--Maruyama scheme (denoted by SEM), see for instance \cite{MR1214374}
$$
X^{\SEM}_{m+1}=X^{\SEM}_m+ f(X^{\SEM}_{m+1})\Delta t + \lambda g(X^{\SEM}_m) \Delta B_{m},\ m=0,\ldots,M-1,
$$
initialised with $X^{\SEM}_{0} = x_{0}$.
\item the tamed Euler scheme (denoted by TE), see for instance \cite{MR2985171}
$$
X^{\TE}_{m+1}=X^{\TE}_m+ f^{M}(X^{\TE}_{m})\Delta t + g^{M}(X^{\TE}_m) \Delta B_{m},\ m=0,\ldots,M-1,
$$
initialised with $X^{\TE}_{0} = x_{0}$, where
$$
f^{M}(x) = \frac{f(x)}{1 + M^{-1/2} |f(x)| + M^{-1/2} | \lambda g(x)|^{2}}
$$
$$
g^{M}(x) = \frac{\lambda g(x)}{1 + M^{-1/2} |f(x)| + M^{-1/2} | \lambda g(x)|^{2}}.
$$
\end{itemize}
Boundary preservation and lack thereof are presented in tables showing, for each scheme, the number of sample paths out of $100$ that only contained values in the invariant domain $\D$. We estimate the $2$-strong convergence orders in Corollary~\ref{cor:LIT_conv} for each considered scheme by computing the following $L^{2}(\Omega)$-error
\begin{equation}\label{eq:L2err}
\left( \E \left[ \sup_{m=0,\ldots,M} \left|X^{approx}_{m} - X^{ref}_{m} \right|^{2} \right] \right)^{1/2},
\end{equation}
where the reference solution $X^{ref}$ is computed using the same scheme as the approximation $X^{approx}$ but with a finer discretisation grid parameter $\Delta t^{ref} = 10^{-7}$, and displaying these errors in loglog plots. The expected value in~\eqref{eq:L2err} is approximated using $300$ Monte Carlo samples, and we have numerically verified that $300$ Monte Carlo samples is sufficient to observe the $2$-strong convergence order. To approximate $I_{(1,1,0)}$, as described in Section~\ref{sec:comp_iter_ints}, we use $M^{aux} = R^{aux} M^{ref} = 32 M^{ref}$ in the implementation.

The particular examples of SDEs that we consider in this section are typical for Lamperti-based schemes, and they are taken from previous work by the author \cite{Ulander2024AB,MR4737060}. We refer to these works for details on what these SDEs model.

\subsection{Allen--Cahn type SDE}
Here we consider the Allen--Cahn SDE given by
\begin{equation}\label{eq:AC_SDE}
\left\lbrace
\begin{aligned}
& \diff X(t) = \left( X(t) - X(t)^{3} \right) \diff t + \lambda \left( 1 - X(t)^{2} \right) \diff B(t),\ t \in (0,T], \\ 
& X(0) = x_{0} \in \D = (-1,1),
\end{aligned}
\right.
\end{equation}
where $\lambda > 0$ is the noise scaling parameter. The coefficient functions $f(r) = r - r^3$ and $g(r) = 1 - r^2$ satisfy Assumptions~\ref{ass:f},~\ref{ass:g}, and~\ref{ass:fg} for any $k \in \mathbb{N}$. Thus, by Corollary~\ref{cor:LIT_conv}, Lamperti--Itô--Taylor-$\gamma$ schemes of arbitrarily high order are applicable.

The transformed SDE is in this case given by
\begin{equation*}
\left\lbrace
\begin{aligned}
& \diff Y(t) = (1 + \lambda^{2}) \Phi^{-1}(Y(t)) \diff t + \lambda \diff B(t),\ t \in (0,T], \\ 
& Y(0) = \Phi(x_{0}) \in \mathbb{R},
\end{aligned}
\right.
\end{equation*}
where
\begin{equation*}
\Phi(r) = \frac{1}{2} \log \left( \frac{1+r}{1-r} \right),\ r \in \D,
\end{equation*}
and
\begin{equation*}
\Phi^{-1}(r) = \frac{e^{2r}-1}{e^{2r}+1},\ r \in \mathbb{R},
\end{equation*}
where we for simplicity choose $w_{0}=0$; that is, $H(r) = (1 + \lambda^{2}) \Phi^{-1}(r)$, for $r \in \mathbb{R}$. We present numerical experiments for $\LIT \gamma$ schemes for $\gamma \in \{ 0.5,1.5,2 \}$ (recall that $\LEM$ and $\LIT1.0$ coincide). Therefore, by Section~\ref{sec:LIT_schemes}, we need to evaluate $H,H',H''$. By direct differentiation, we have that
\begin{equation*}
H'(r) = (1 + \lambda^{2}) g(\Phi^{-1}(r)),\ r \in \mathbb{R},
\end{equation*}
and that
\begin{equation*}
H''(r) = (1 + \lambda^{2}) g'(\Phi^{-1}(r)) g(\Phi^{-1}(r)),\ r \in \mathbb{R}.
\end{equation*}

In Figure~\ref{num:AC_path_comp}, we present sample paths of the considered numerical schemes applied to the Allen--Cahn SDE in~\eqref{eq:AC_SDE}, with the same Brownian motion sample path used for all schemes. Figure~\ref{num:AC_path_comp} clearly show that the comparison schemes $\EM$, $\SEM$, and $\TE$ leave the invariant domain $\D = (-1,1)$ of~\eqref{eq:AC_SDE}, and are therefore not boundary-preserving. In contrast, Figure~\ref{num:AC_path_comp} also shows that the sample paths of the $\LIT \gamma$ schemes do not leave the invariant domain $\D = (-1,1)$, confirming that the $\LIT \gamma$ schemes are boundary-preserving. Next, in Table~\ref{tb:AC}, we present, for each considered scheme, the proportion out of $100$ samples that only produced values inside the invariant domain $\D = (-1,1)$. Table~\ref{tb:AC} numerically verify that the $\LIT \gamma$ schemes, for $\gamma \in \{0.5,1.5,2 \}$, are boundary-preserving and that the comparison schemes $\EM$, $\SEM$, and $\TE$ are not boundary-preserving. 

\begin{figure}[htp]
\begin{center}
  \includegraphics[width=0.8\textwidth]{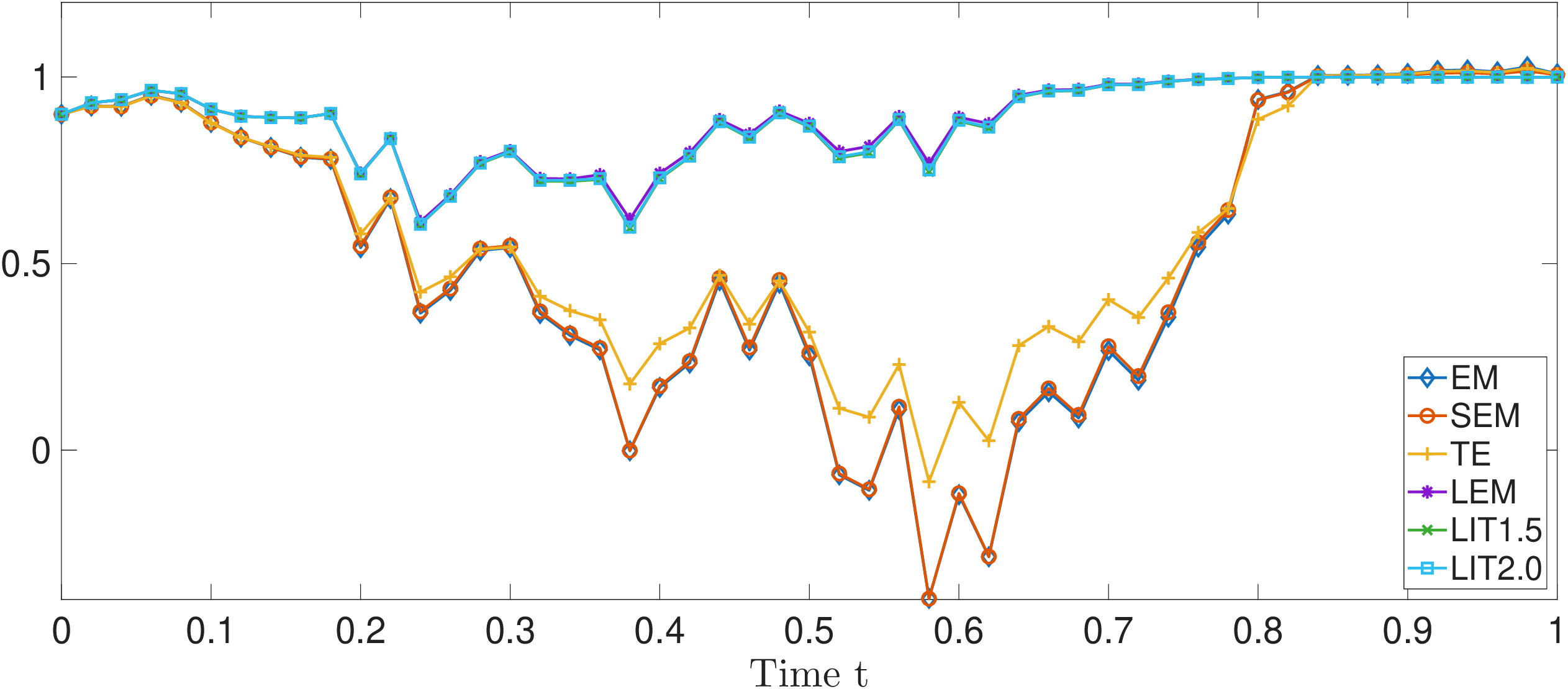}
  \caption{Path comparison of the Lamperti--Euler--Maruyama ($\LEM$) scheme, Lamperti--Itô--Taylor-1.5 ($\LIT 1.5$) scheme, Lamperti--Itô--Taylor-2.0 ($\LIT 2.0$) scheme, Euler--Maruyama ($\EM$) scheme, the semi-implicit Euler--Maruyama ($\SEM$) scheme, and the tamed Euler ($\TE$) scheme for the Allen--Cahn SDE in~\eqref{eq:AC_SDE} with parameters $\lambda = 2$, $x_{0}=0.9$, $T = 0.4$ and $M=50$.}\label{num:AC_path_comp}
  \end{center}
\end{figure}

\begin{table}[!htbp]
\begin{center}
\begin{tabular}{||c c c c c c c||}
 \hline
 $\lambda$ & $\LEM$ & $\LIT 1.5$ & $\LIT 2.0$ & $\EM$ & $\SEM$ & $\TE$ \\ [0.5ex]
 \hline\hline
 $1$ & $100/100$ & $100/100$ & $100/100$ & $99/100$ & $99/100$ & $99/100$\\
 \hline
 $2$ & $100/100$ & $100/100$ & $100/100$ & $23/100$ & $28/100$ & $42/100$ \\
 \hline
 $3$ & $100/100$ & $100/100$ & $100/100$ & $0/100$  & $0/100$& $13/100$ \\ [1ex]
 \hline
\end{tabular}
\caption{Proportion of samples containing only values in $(-1,1)$ out of $100$ simulated sample paths for the Lamperti--Euler--Maruyama ($\LEM$) scheme, Lamperti--Itô--Taylor-1.5 ($\LIT 1.5$) scheme, Lamperti--Itô--Taylor-2.0 ($\LIT 2.0$) scheme, Euler--Maruyama ($\EM$) scheme, the semi-implicit Euler--Maruyama ($\SEM$) scheme, and the tamed Euler ($\TE$) scheme for the Allen--Cahn SDE in~\eqref{eq:AC_SDE} for three choices of $\lambda>0$. The parameters used are: $T=1$, $\Delta t = 1/50$ and with $x_{0}$ uniformly distributed on $\D = (-1,1)$ for each sample. \label{tb:AC}}
\end{center}
\end{table}

Finally, in Figure~\ref{num:AC_conv_plot}, we present the $L^{2}(\Omega)$-errors of the $\LIT \gamma$ schemes for $\gamma \in \{0.5,1.5,2 \}$ together with reference lines with slopes $1,1.5$, and $2$, respectively. Figure~\ref{num:AC_conv_plot} numerically verifies the convergence result in Corollary~\ref{cor:LIT_conv}.

\begin{figure}[htp]
\begin{center}
  \includegraphics[width=0.8\textwidth]{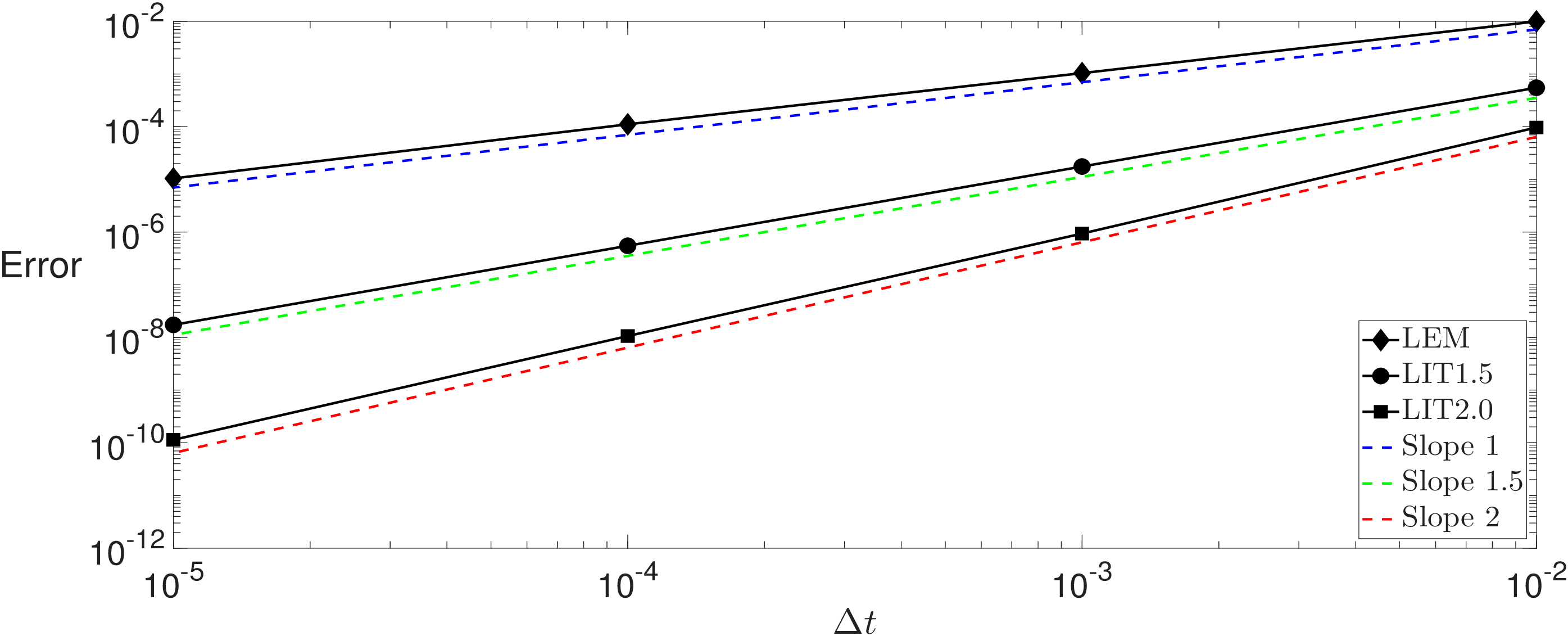}
  \caption{$L^{2}(\Omega)$-errors on the interval $[0,1]$ of the Lamperti--Euler--Maruyama ($\LEM$) scheme, the Lamperti--Itô--Taylor-$1.5$ ($\LIT 1.5$) scheme, and the Lamperti--Itô--Taylor-$2.0$ ($\LIT 2.0$) scheme for the Allen--Cahn SDE in~\eqref{eq:AC_SDE} for $\lambda=1$ and reference lines with slopes $1$, $1.5$, and $2$. Averaged over $300$ samples and $x_{0}=0$.}\label{num:AC_conv_plot}
  \end{center}
\end{figure}

\subsection{Nagumo-type SDE}
Here we consider the Nagumo-type SDE given by
\begin{equation}\label{eq:Nagumo_SDE}
\left\lbrace
\begin{aligned}
& \diff X(t) = X(t) \left( 1 - X(t) \right) \left( X(t) - \eta \right) \diff t + X(t) \left( 1 - X(t) \right) \diff B(t),\ t \in (0,T], \\ 
& X(0) = x_{0} \in \D = (0,1),
\end{aligned}
\right.
\end{equation}
for some $\eta \in (0,1)$. The coefficient functions $f(r) = r (1-r)(r-\eta)$ and $g(r) = r (1-r)$ satisfy Assumptions~\ref{ass:f},~\ref{ass:g}, and~\ref{ass:fg} for any $k \in \mathbb{N}$. Thus, by Corollary~\ref{cor:LIT_conv}, Lamperti--Itô--Taylor-$\gamma$ schemes of arbitrarily high order are applicable.

The transformed SDE is in this case given by
\begin{equation*}
\left\lbrace
\begin{aligned}
& \diff Y(t) = \left(- \left( a + \frac{\lambda^{2}}{2} \right) + (1+\lambda^{2}) \Phi^{-1}(Y(t)) \right) \diff t + \lambda \diff B(t),\ t \in (0,T], \\ 
& Y(0) = \Phi(x_{0}) \in \mathbb{R},
\end{aligned}
\right.
\end{equation*}
where
\begin{equation*}
\Phi(r) = \log \left( \frac{r}{1-r} \right),\ r \in \D,
\end{equation*}
and
\begin{equation*}
\Phi^{-1}(r) = \frac{e^{r}}{e^{r} + 1},\ r \in \mathbb{R},
\end{equation*}
where we for simplicity choose $w_{0} = 1/2$; that is $H(r) = -\left(a + \frac{\lambda^{2}}{2} \right) + (1 + \lambda^{2}) \Phi^{-1}(r)$, for $r \in \mathbb{R}$. We need to evaluate $H,H',H''$, since we present numerical experiments for $\LIT \gamma$ schemes for $\gamma \in \{0.5,1.5,2 \}$.

Direct differentiation yields
\begin{equation*}
H'(r) = (1+\lambda^{2}) g(\Phi^{-1}(r)),\ r \in \mathbb{R},
\end{equation*}
and
\begin{equation*}
H''(r) = (1+\lambda^{2}) g'(\Phi^{-1}(r)) g(\Phi^{-1}(r)),\ r \in \mathbb{R}.
\end{equation*}

We first, in Figure~\ref{num:Nagumo_path_comp}, present sample paths of the considered schemes applied to the Nagumo SDE in~\eqref{eq:Nagumo_SDE} using the same Brownian motion sample paths for all schemes. Figure~\ref{num:Nagumo_path_comp} shows that the comparison schemes $\EM$, $\SEM$, and $\TE$ leave the invariant domain $\D = (0,1)$ while the $\LIT \gamma$ do not leave the invariant domain $\D=(0,1)$. To further strengthen the conclusion that the $\LIT \gamma$ schemes are boundary-preserving and that the comparison schemes are not boundary-preserving, we present in Table~\ref{tb:Nagumo} the proportion out of $100$ samples for each considered scheme that only contained values in the invariant domain $\D = (0,1)$. Figure~\ref{num:Nagumo_path_comp} and Table~\ref{tb:Nagumo} numerically confirm that the $\LIT \gamma$ schemes are boundary-preserving and that the comparison schemes are not boundary-preserving.

\begin{figure}[htp]
\begin{center}
  \includegraphics[width=0.8\textwidth]{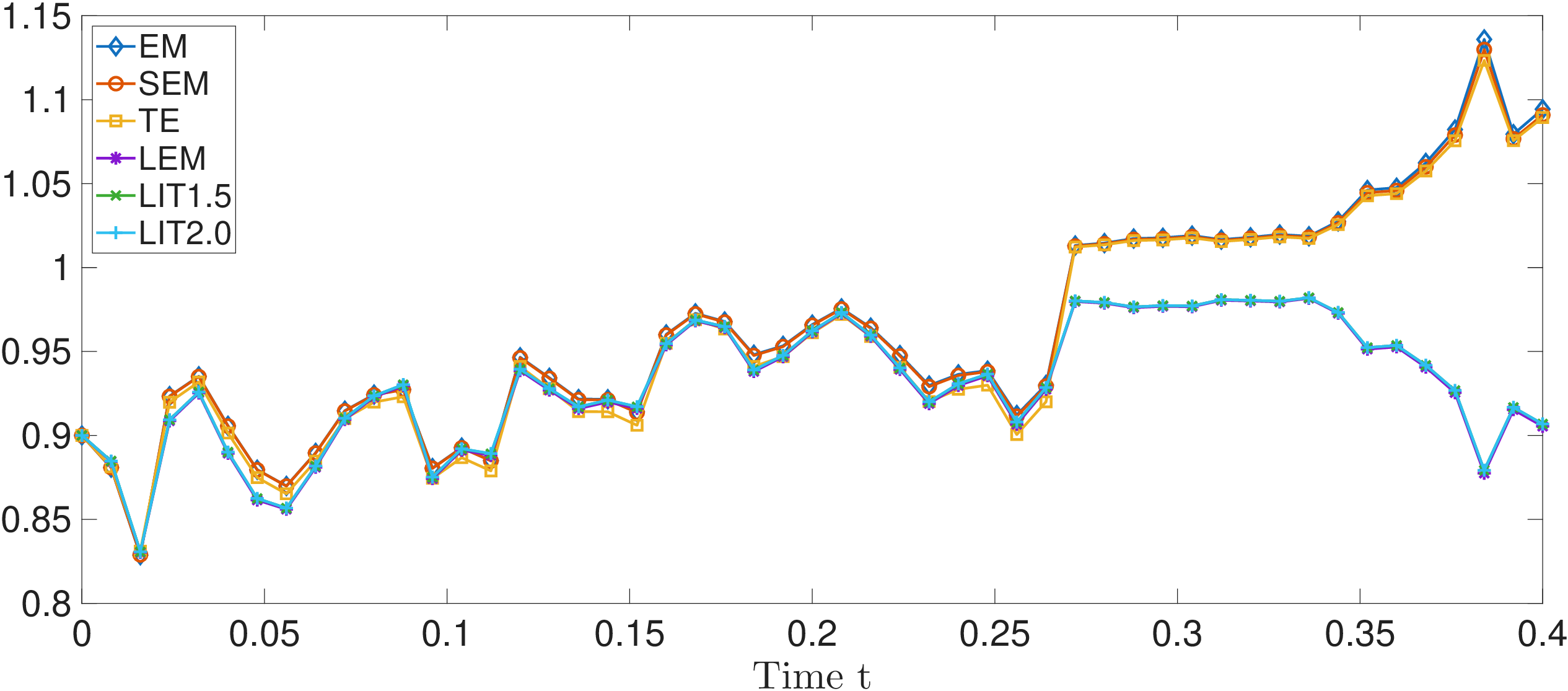}
  \caption{Path comparison of the Lamperti--Euler--Maruyama ($\LEM$) scheme, Lamperti--Itô--Taylor-1.5 ($\LIT 1.5$) scheme, Lamperti--Itô--Taylor-2.0 ($\LIT 2.0$) scheme, Euler--Maruyama ($\EM$) scheme, the semi-implicit Euler--Maruyama ($\SEM$) scheme, and the tamed Euler ($\TE$) scheme applied to the Nagumo SDE in~\eqref{eq:Nagumo_SDE} with parameters $\gamma= 1/4, \lambda = 4$, $x_{0}=0.9$, $T = 0.4$ and $M=50$.}\label{num:Nagumo_path_comp}
  \end{center}
\end{figure}

\begin{table}[!htbp]
\begin{center}
\begin{tabular}{||c c c c c c c||}
 \hline
 $\lambda$ & $\LEM$ & $\LIT 1.5$ & $\LIT 2.0$ & $\EM$ & $\SEM$ & $\TE$ \\ [0.5ex]
 \hline\hline
 $2$ & $100/100$ & $100/100$ & $100/100$ & $100/100$ & $100/100$ & $100/100$\\
 \hline
 $3$ & $100/100$ & $100/100$ & $100/100$ & $73/100$ & $75/100$ & $74/100$ \\
 \hline
 $4$ & $100/100$ & $100/100$ & $100/100$ & $21/100$  & $21/100$& $25/100$ \\ [1ex]
 \hline
\end{tabular}
\caption{Proportion of samples containing only values in $\D = (0,1)$ out of $100$ simulated sample paths for the Lamperti--Euler--Maruyama ($\LEM$) scheme, Lamperti--Itô--Taylor-1.5 ($\LIT 1.5$) scheme, Lamperti--Itô--Taylor-2.0 ($\LIT 2.0$) scheme, Euler--Maruyama ($\EM$) scheme, the semi-implicit Euler--Maruyama ($\SEM$) scheme, and the tamed Euler ($\TE$) scheme for the Nagumo SDE in~\eqref{eq:Nagumo_SDE} for three choices of $\lambda>0$. The parameters used are: $\gamma=1/4$, $T=1$, $\Delta t = 1/50$ and with $x_{0}$ uniformly distributed on $\D = (0,1)$ for each sample. \label{tb:Nagumo}}
\end{center}
\end{table}

In Figure~\ref{num:Nagumo_conv_plot}, we present the $L^{2}(\Omega)$-errors of the $\LIT \gamma$ schemes for $\gamma \in \{0.5,1.5,2 \}$ together with reference lines with slopes $1,1.5$, and $2$, respectively. The $L^{2}(\Omega)$-errors lines and the references lines in Figure~\ref{num:Nagumo_conv_plot} align well, numerically confirming the convergence result in Corollary~\ref{cor:LIT_conv}.

\begin{figure}[htp]
\begin{center}
  \includegraphics[width=0.8\textwidth]{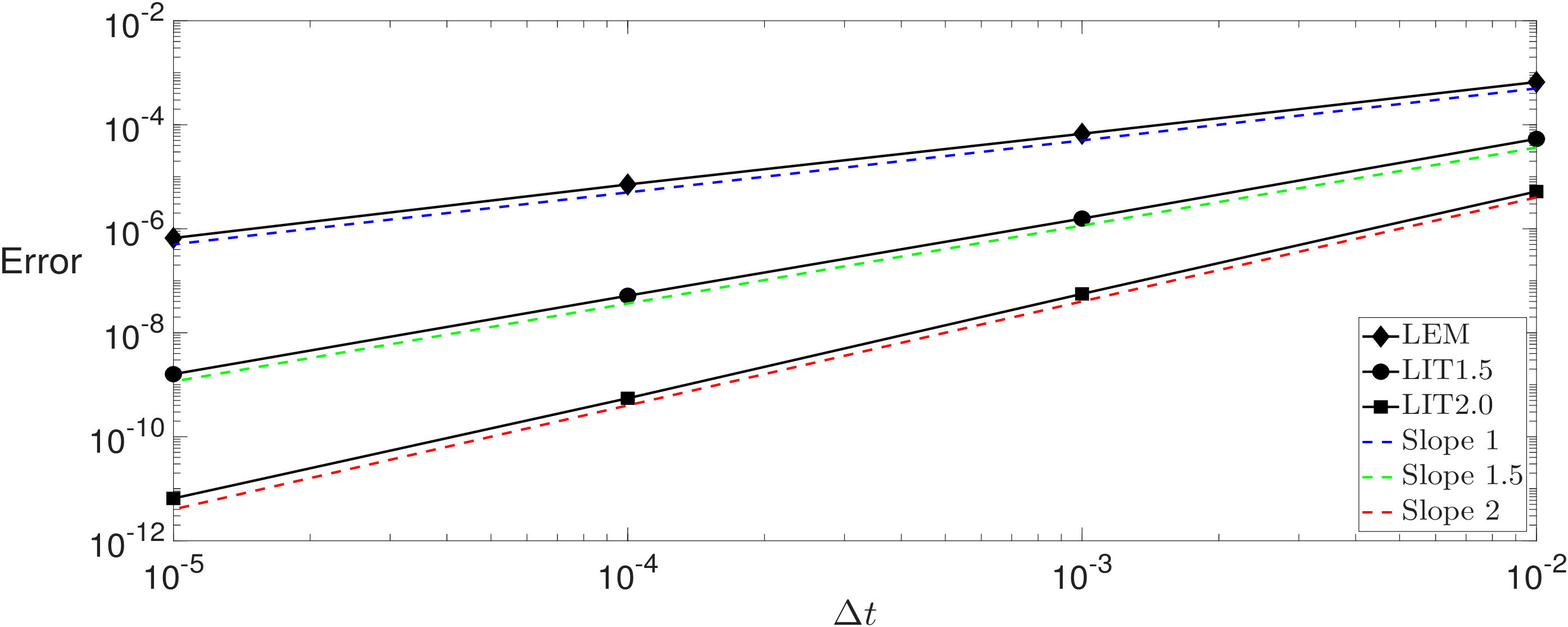}
  \caption{$L^{2}(\Omega)$-errors on the interval $[0,1]$ of the Lamperti--Euler--Maruyama ($\LEM$) scheme, the Lamperti--Itô--Taylor-$1.5$ ($\LIT 1.5$) scheme, and the Lamperti--Itô--Taylor-$2.0$ ($\LIT 2.0$) scheme for the Nagumo SDE in~\eqref{eq:Nagumo_SDE} for $\lambda=1$ and reference lines with slopes $1$, $1.5$, and $2$, respectively. Averaged over $3100$ samples and $x_{0} = 1/2$.}\label{num:Nagumo_conv_plot}
  \end{center}
\end{figure}

\subsection{SIS SDE}
Here we consider the SIS SDE given by
\begin{equation}\label{eq:SIS_SDE}
\left\lbrace
\begin{aligned}
& \diff X(t) = X(t) \left( 1 - X(t) \right) \diff t + \lambda X(t) \left( 1 - X(t) \right) \diff B(t),\ t \in (0,T], \\ 
& X(0) = x_{0} \in \D = (0,1).
\end{aligned}
\right.
\end{equation}
The coefficient functions $f(r) = r(1-r)$ and $g(r) = r(1-r)$ satisfy Assumptions~\ref{ass:f},~\ref{ass:g}, and~\ref{ass:fg} for any $k \in \mathbb{N}$. Thus, by Corollary~\ref{cor:LIT_conv}, Lamperti--Itô--Taylor-$\gamma$ schemes of arbitrarily high order are applicable.

The transformed SDE is in this case given by
\begin{equation*}
\left\lbrace
\begin{aligned}
& \diff Y(t) = \left(1 - \frac{\lambda^{2}}{2} + \lambda^{2} \Phi^{-1}(Y(t)) \right) \diff t + \lambda \diff B(t),\ t \in (0,T], \\ 
& Y(0) = \Phi(x_{0}) \in \mathbb{R},
\end{aligned}
\right.
\end{equation*}
where
\begin{equation*}
\Phi(r) = \log \left( \frac{r}{1-r} \right),\ r \in \D,
\end{equation*}
and
\begin{equation*}
\Phi^{-1}(r) = \frac{e^{x}}{e^{r} + 1},\ r \in \mathbb{R},
\end{equation*}
where we for simplicity choose $w_{0}=1/2$; that is,
\begin{equation*}
H(r) = 1 - \frac{\lambda^{2}}{2} + \lambda^{2} \Phi^{-1}(r),\ r \in \mathbb{R}.
\end{equation*}
As in the two previous numerical experiments, we need to evaluate $H,H'$ and $H''$ to implement $\LEM, \LIT 1.5$ and $\LIT 2.0$.

Direct differentiation gives us
\begin{equation*}
H'(r) = \lambda^{2} g(\Phi^{-1}(r)),\ r \in \mathbb{R},
\end{equation*}
and
\begin{equation*}
H''(r) = \lambda^{2} g'(\Phi^{-1}(r)) g(\Phi^{-1}(r)),\ r \in \mathbb{R}.
\end{equation*}

We first showcase in Figure~\ref{num:SIS_path_comp} that the comparison schemes $\EM$, $\SEM$, and $\TE$ leave the invariant domain $\D = (0,1)$ of the SIS SDE~\eqref{eq:SIS_SDE}, and are hence not boundary-preserving, while the $\LIT \gamma$ schemes, for $\gamma \in \{0.5,1.5,2 \}$, are confined to the invariant domain $\D = (0,1)$. We illustrate this more systematically in Table~\ref{tb:SIS}, where we display the proportion of samples out of $100$ that only contained values in the invariant domain $\D = (0,1)$ for the comparison schemes $\EM$, $\SEM$, and $\TE$ and for the $\LIT \gamma$ schemes. Table~\ref{tb:SIS} further illustrates that the comparison schemes are not boundary-preserving and strengthens the hypothesis that the $\LIT \gamma$ schemes, for $\gamma \in \{0.5,1.5,2 \}$, are boundary-preserving.

\begin{figure}[htp]
\begin{center}
  \includegraphics[width=0.8\textwidth]{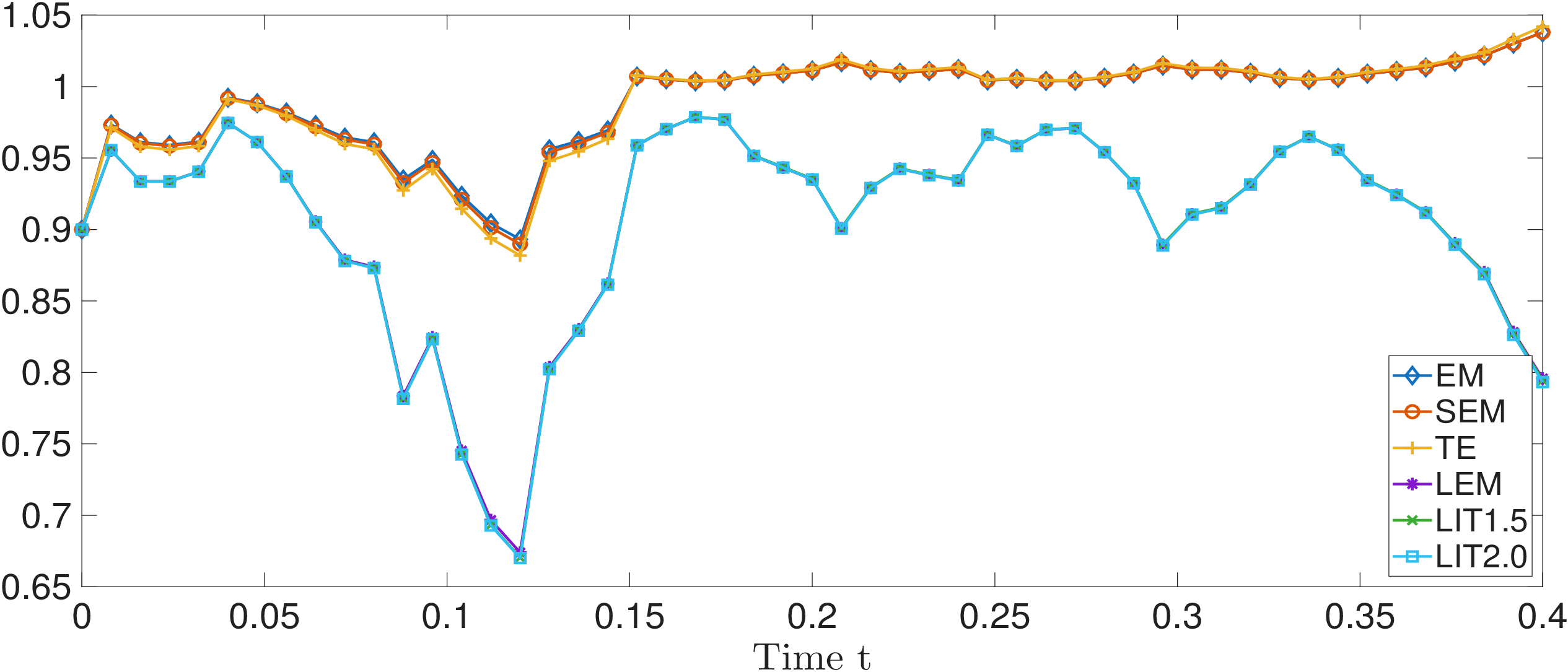}
  \caption{Path comparison of the Lamperti--Euler--Maruyama ($\LEM$) scheme, Lamperti--Itô--Taylor-1.5 ($\LIT 1.5$) scheme, Lamperti--Itô--Taylor-2.0 ($\LIT 2.0$) scheme, Euler--Maruyama ($\EM$) scheme, the semi-implicit Euler--Maruyama ($\SEM$) scheme, and the tamed Euler ($\TE$) scheme applied to the SIS SDE in~\eqref{eq:SIS_SDE} with parameters $\lambda = 4$, $x_{0}=0.9$, $T = 0.4$ and $M=50$.}\label{num:SIS_path_comp}
  \end{center}
\end{figure}

\begin{table}[!htbp]
\begin{center}
\begin{tabular}{||c c c c c c c||}
 \hline
 $\lambda$ & $\LEM$ & $\LIT 1.5$ & $\LIT 2.0$ & $\LEM$ & $\SEM$ & $\TE$ \\ [0.5ex]
 \hline\hline
 $3$ & $100/100$ & $100/100$ & $100/100$ & $100/100$ & $100/100$ & $100/100$\\
 \hline
 $4$ & $100/100$ & $100/100$ & $100/100$ & $25/100$ & $23/100$ & $29/100$ \\
 \hline
 $5$ & $100/100$ & $100/100$ & $100/100$ & $5/100$  & $4/100$& $5/100$ \\ [1ex]
 \hline
\end{tabular}
\caption{Proportion of samples containing only values in $\D = (0,1)$ out of $100$ simulated sample paths for the Lamperti--Euler--Maruyama ($\LEM$) scheme, Lamperti--Itô--Taylor-1.5 ($\LIT 1.5$) scheme, Lamperti--Itô--Taylor-2.0 ($\LIT 2.0$) scheme, Euler--Maruyama ($\EM$) scheme, the semi-implicit Euler--Maruyama ($\SEM$) scheme, and the tamed Euler ($\TE$) scheme for the SIS SDE in~\eqref{eq:SIS_SDE} for three choices of $\lambda>0$. The parameters used are: $T=1$, $\Delta t = 1/50$ and with $x_{0}$ uniformly distributed on $\D = (0,1)$ for each sample. \label{tb:SIS}}
\end{center}
\end{table}

Lastly, we display the $L^{2}(\Omega)$-errors of the $\LIT \gamma$ schemes, for $\gamma \in \{0.5,1.5,2 \}$, applied to the SIS SDE~\eqref{eq:SIS_SDE} in Figure~\ref{num:SIS_conv_plot}. The decay of the $L^{2}(\Omega)$-errors in Figure~\ref{num:SIS_conv_plot} align well with the reference lines with slopes $1,1.5$, and $2$, respectively, and numerically confirm the convergence result in Corollary~\ref{cor:LIT_conv}.

\begin{figure}[htp]
\begin{center}
  \includegraphics[width=0.8\textwidth]{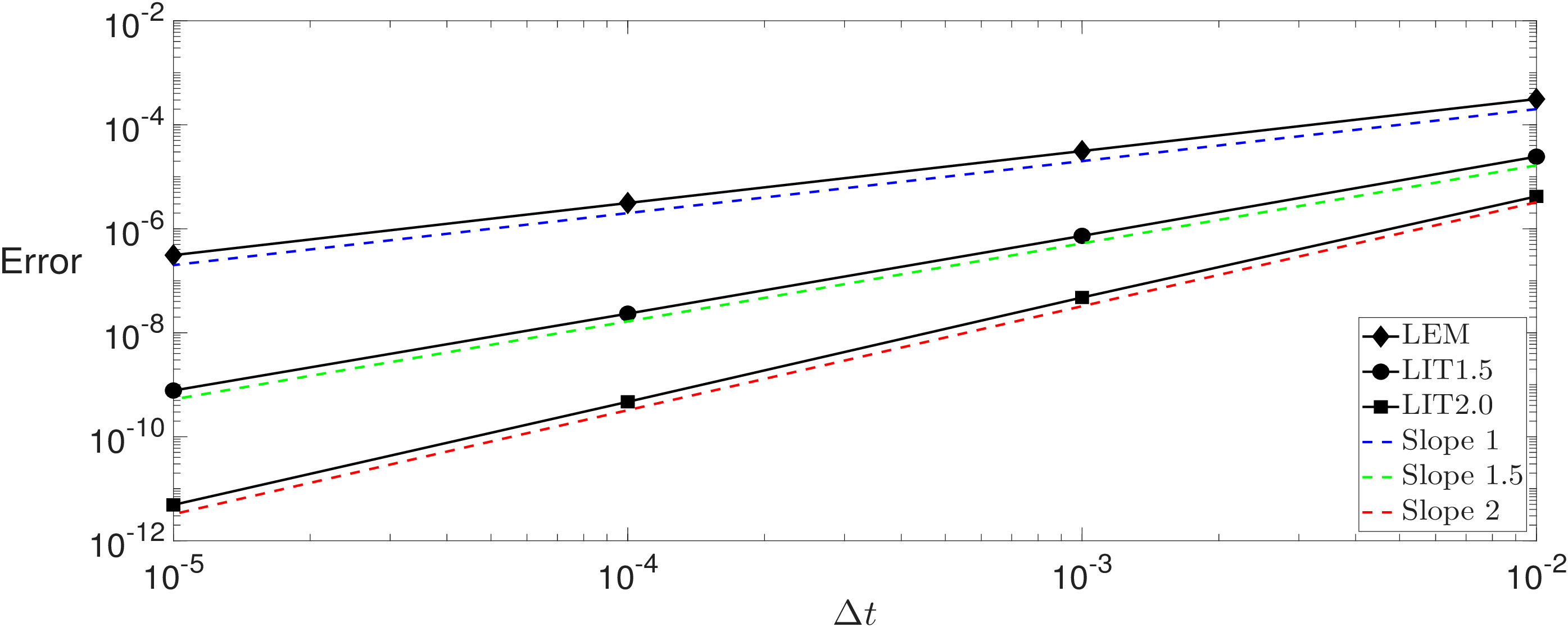}
  \caption{$L^{2}(\Omega)$-errors on the interval $[0,1]$ of the Lamperti--Euler--Maruyama ($\LEM$) scheme, the Lamperti--Itô--Taylor-$1.5$ ($\LIT 1.5$) scheme, and the Lamperti--Itô--Taylor-$2.0$ ($\LIT 2.0$) scheme for the SIS SDE in~\eqref{eq:SIS_SDE} for $\lambda=1$ and reference lines with slopes $1$, $1.5$, and $2$. Averaged over $300$ samples and $x_{0} = 1/2$.}\label{num:SIS_conv_plot}
  \end{center}
\end{figure}

\section{Discussion}
The main contribution of this work is to combine the Lamperti transform with suitable regularity assumptions on the coefficient functions of the original SDE that ensure high regularity of the transformed equation. This allows high-order numerical methods to be applied to the transformed SDE, after which the inverse Lamperti transform yields a boundary-preserving high-order scheme for the original equation. Although the present work focuses on strong Itô--Taylor schemes, the same strategy is not limited to this class of methods. In principle, it can also be combined with other high-order strong schemes, such as high-order time-splitting methods and stochastic Runge--Kutta schemes, as well as with high-order weak schemes, including weak Itô--Taylor methods.

\section*{Acknowledgements}
This work was made possible by the Swedish Defence Research Agency (FOI).

\bibliographystyle{abbrv}
\bibliography{Boundary-preserving_Lamperti-Ito-Taylor_approximations_for_some_stochastic_differential_equations}

\end{document}